\documentclass[reqno,12pt]{amsproc}

\usepackage{amsbsy}
\usepackage{amsfonts}
\usepackage{amsmath}
\usepackage{amsthm}
\usepackage{amssymb}
\usepackage{enumerate}

\usepackage{hyperref}
\usepackage{color}
\usepackage{etoolbox}
\expandafter\patchcmd\csname\string\proof\endcsname
  {\normalparindent}{0pt }{}{}
\usepackage[margin=1.5in]{geometry}

\begin{document}

\title[Tail bounds for counts of zeros]{Tail bounds for counts of zeros and eigenvalues, and an application to ratios}

\author{Brad Rodgers}

\address{Institut f\"{u}r Mathematik, Universit\"{a}t Z\"{u}rich, Winterthurerstr. 190, CH-8057 Z\"{u}rich}
\email{rbrad@umich.edu}

\maketitle

\theoremstyle{plain}

\newtheorem{theorem}{Theorem}[section]
\newtheorem{lem}[theorem]{Lemma}
\newtheorem{prop}[theorem]{Proposition}
\newtheorem{cor}[theorem]{Corollary}
\newtheorem{conj}[theorem]{Conjecture}
\newtheorem{dfn}[theorem]{Definition}

\newcommand\MOD{\textrm{ (mod }}
\newcommand\Var{\mathrm{Var}}
\newcommand\spann{\mathrm{span}}
\newcommand\Prob{\mathop{\vcenter{\hbox{\relsize{+2}$\mathbf{P}$}}}}
\newcommand\Tr{\mathrm{Tr}}
\newcommand\Ss{\mathcal{S}}
\newcommand\Dd{\mathcal{D}}
\newcommand\Uu{\mathcal{U}}
\newcommand\supp{\mathrm{supp}\;}
\newcommand\sgn{\mathrm{sgn}}
\newcommand\bb{\mathbb}
\newcommand\pvint{-\!\!\!\!\!\!\!\int_{-\infty}^\infty}
\def\res{\mathop{\mathrm{Res}}}
\newcommand\Z{\mathcal{Z}}
\newcommand\E{\mathcal{E}}
\newcommand\law{\mathrel{\overset{\makebox[0pt]{\mbox{\normalfont\tiny\sffamily law}}}{=}}}

\begin{abstract}
Let $t$ be random and uniformly distributed in the interval $[T,2T]$, and consider the quantity $N(t+1/\log T) - N(t)$, a count of zeros of the Riemann zeta function in a box of height $1/\log T$. Conditioned on the Riemann hypothesis, we show that the probability this count is greater than $x$ decays at least as quickly as $e^{-Cx\log x}$, uniformly in $T$. We also prove a similar results for the logarithmic derivative of the zeta function, and likewise analogous results for the eigenvalues of a random unitary matrix.

We use results of this sort to show on the Riemann hypothesis that the averages
$$
\frac{1}{T} \int_T^{2T} \Bigg| \frac{\zeta\Big(\frac{1}{2} + \frac{\alpha}{\log T} + it\Big)}{\zeta\Big(\frac{1}{2}+ \frac{\beta}{\log T} + it\Big)}\Bigg|^m\,dt
$$
remain bounded as $T\rightarrow\infty$, for $\alpha, \beta$ complex numbers with $\beta\neq 0$. Moreover we show rigorously that the local distribution of zeros asymptotically controls ratio averages like the above; that is, the GUE Conjecture implies a (first-order) ratio conjecture. 
\end{abstract}

\section{Introduction}
\label{1}

\subsection{}
This paper is comprised of two parts. In the first part we prove, conditioned on the Riemann hypothesis (RH), that local linear statistics of the zeros of the Riemann zeta function have uniformly sub-exponential tails. More precisely, label the non-trivial zeros of the zeta function $1/2+i\gamma,$ with $\gamma\in\mathbb{R}$. 
We prove the following theorem.

\begin{theorem}[Tail bound for zeros]
\label{quad_bound} 
Assume RH. Define $Q(\xi):=1/(1+\xi^2)$. Then for all $x \geq 2$ and all $T \geq 2$.
$$
\frac{1}{T}\,\mathrm{meas}\Big\{ t\in [T,2T]:\, \sum_\gamma Q\Big(\frac{\log T}{2\pi}(\gamma-t)\Big) \geq x\Big\} \ll e^{-C x\log x},
$$
where the constant $C$ and the implicit constant are absolute.
\end{theorem}

Here and in what follows, zeros are counted with multiplicity (in the unlikely event that some zero is not simple).

To elaborate on the meaning of this result: the ordinates $\gamma$ have density $\log T/2\pi$ near a height $T$, and for $t\in [T,2T]$, the points $\{\tfrac{\log T}{2\pi}(\gamma-t)\}$ are spaced so as to have a density of roughly $1$, at least for $\gamma$ near $t$. Theorem \ref{quad_bound} therefore bounds the frequency with which these respaced zeros can occur in large clumps. The theorem is only of interest when $x$ is large.

Plainly Theorem \ref{quad_bound} also implies the same estimate when $Q$ is replaced by any function $\eta$ that decays quadratically (with constants depending on $\eta$). Letting $\eta = \mathbf{1}_{[0,1/2\pi]}$, and defining as usual $N(T):= \#\{\gamma: \gamma\in(0,T)\}$, we obtain a corollary that may be easier to understand at a glance.

\begin{cor}
\label{interval_bound}
Assume RH. For all $x\geq 2$ and all $T\geq 2$,
$$
\frac{1}{T}\,\mathrm{meas}\big\{ t\in [T,2T]:\, N(t+1/\log T)-N(t) \geq x\big\} \ll e^{-C x\log x}
$$
where the constant $C$ and the implicit constant are absolute.
\end{cor}

\textbf{Remark:} This result refines a moment bound of Fujii \cite[Main Theorem]{Fu}, and is closely related, even in the method of its proof, to a bound of Soundararajan \cite[Theorem 2]{So2}, who proves estimates of a similar strength, but in which $x$ grows with $T$, and in which the size of the interval may grow at a faster rate than $1/\log T$.

We note that without assuming RH, it is possible to prove an upper bound $e^{-cx}$, where $c$ is an absolute constant.

We also develop in Theorems \ref{J_bound} and Lemma \ref{W_markov} estimates for more oscillatory counts of zeros. As a consequence we obtain an upper bound for the logarithmic derivative of the zeta function.

\begin{theorem} [Tail bounds for $\zeta'/\zeta$]
\label{log_bound} 
Assume RH, and fix $\alpha > 0$. For $x \geq 2$ and $T \geq 2$,
$$
\frac{1}{T}\,\mathrm{meas}\Big\{ t\in [T,2T]: \frac{1}{\log T} \Big|\frac{\zeta'}{\zeta}\Big(\frac{1}{2}+\frac{\alpha}{\log T} + it\Big)\Big| \geq x \Big\} \ll e^{-C x \log x},
$$
where the constant $C$ and the implied constant depend only on $\alpha$. 
\end{theorem}

\textbf{Remark:} This strengthens moment bounds for the logarithmic derivative of the zeta function, which have been proved under RH and some additional hypotheses by Farmer, Gonek, Lee, and Lester \cite[Corollary 2.1]{FaGoLeLe}, and subsequently under RH alone in the author's thesis (see \cite[Theorem 2.1]{Ro}).

We apply the tail bound, Theorem \ref{quad_bound}, and these other bounds to consider averages of ratios of the zeta function. We develop an upper bound for these averages.

\begin{theorem}[Moment bound for ratios]
\label{ratio_bound}
Assume RH. For any $\alpha, \beta \in \mathbb{C}$ with $\Re\, \beta \neq 0$, and for any $m \geq 0$, uniformly for $T\geq 2,$
$$
\frac{1}{T} \int_T^{2T} \Bigg| \frac{\zeta\Big(\frac{1}{2} + \frac{\alpha}{\log T} + it\Big)}{\zeta\Big(\frac{1}{2}+ \frac{\beta}{\log T} + it\Big)}\Bigg|^m\,dt \ll_{\alpha, \beta, m} 1.
$$
\end{theorem}

\subsection{}
The second part of the paper requires some knowledge from random matrix theory. Before all else, we will develop bounds for counts of eigenvalues of random unitary matrices analogous to those above for zeta zeros.

Moreover, we show rigorously that the asymptotic evaluation of averages of the sort considered Theorem \ref{ratio_bound} follow from knowing the local distribution of zeros of the zeta function. Recall the following well-known conjecture about the local distribution of zeros.

\begin{conj}[GUE Conjecture]
\label{GUEConj}
Assume RH. For all fixed $k$ and continuous and quadratically decaying\footnote{By quadratically decaying, we mean $\eta(x) = O\Big(\frac{1}{1+x_1^2}\cdots \frac{1}{1+x_k^2}\Big).$ A purist may object that it is more natural to make this conjecture for only compactly supported $\eta$, but these two versions of this conjecture may be seen without too much effort to be equivalent.} test functions $\eta: \bb R^k \rightarrow \bb R$,
\small
\begin{equation*}
\frac{1}{T}\int_T^{2T} \sum_{\substack{\gamma_1,...,\gamma_k\\\mathrm{distinct}}} \eta\Big(\tfrac{\log T}{2\pi}(\gamma_1-t),...,\tfrac{\log T}{2\pi}(\gamma_k-t)\Big)\,dt \sim \int_{\bb R^k} \eta(x) \, \det_{k\times k}\big(K(x_i-x_j)\big)\,d^k x,
\end{equation*}
\normalsize
as $T\rightarrow\infty$, where the $ij^\mathrm{th}$ entry of the $k\times k$ determinant is given by $K(x_i-x_j) = \sin \pi (x_i-x_j) / \pi(x_i-x_j)$.
\end{conj}

We also recall a conjecture for the \emph{first order} asymptotics of ratios of the zeta function.

\begin{conj}[Local Ratios Conjecture with real translations]
\label{RatiosConj}
Assume RH. For all fixed $k\geq 1$ and all fixed collections of numbers $\alpha_1,...,\alpha_m,\beta_1,...,\beta_m\in \mathbb{R}$, with $ \beta_\ell\neq 0$ for all $\ell$, and\footnote{We clearly lose no generality from this restriction.} $\alpha_i\neq \beta_j$ for all $i,j$, we have
\begin{equation}
\label{EqRatios}
\frac{1}{T} \int_T^{2T} \prod_{\ell=1}^m\frac{\zeta\Big(\frac{1}{2} + \frac{\alpha_\ell}{\log T} + it\Big)}{\zeta\Big(\frac{1}{2}+ \frac{\beta_\ell}{\log T} + it\Big)}\,dt \sim \frac{\det\Big(\frac{E(\alpha_i,\beta_j)}{\alpha_i-\beta_j}\Big)}{\det\Big(\frac{1}{\alpha_i-\beta_j}\Big)},
\end{equation}
where
$$
E(\alpha,\beta):= \begin{cases} e^{-\alpha+\beta} & \Re\, \beta < 0 \\ 1 & \Re\, \beta > 0.\end{cases}
$$
\end{conj}

As an application of the techniques above, we show that the first of these claims implies the second.

\begin{theorem} 
\label{corr_implies_ratios}
The GUE Conjecture implies the Local Ratios Conjecture with real translations.
\end{theorem}

There is a seemingly more general conjecture than Conjecture \ref{RatiosConj} in which $\alpha_1,...,\alpha_m,$ $\beta_1,...,\beta_m$ are allowed to lie in $\mathbb{C}$, with $\Re\, \beta_\ell \neq 0 $ for all $\ell$. Such a conjecture may be called just the \emph{Local Ratios Conjecture}. 

This increase in generality is really only apparent. It is possible using similar methods to see that the GUE Conjecture also implies the Local Ratios Conjecture, for general $\alpha$ and $\beta$. The proof of this claim requires a somewhat more lengthy technical argument, so we will not prove it here. We will instead say only a few words about what modifications in the proof of Theorem \ref{corr_implies_ratios} are necessary for it at the end of this paper.

\subsection{}
The study of the average of ratios of the zeta function has a long history. Conjecture \ref{RatiosConj} was first put forward in the case $m=2$ by Farmer \cite{Fa}, who understood it was closely connected with the local distribution of zeros of the Riemann zeta function. Farmer showed that the $m=2$ case of (a uniform version of) what we have called the Local Ratios Conjecture implies the $k=2$ case (pair correlation) of the GUE Conjecture \cite{Fa2}, and later produced similar implications for the $m=3, k=3$ case, while even higher correlations may be obtained from the work in \cite{CoSn2}. To our knowledge the present paper is the first rigorous work in the opposite direction.

More recently, a flurry of work has centered around the average of such ratios when the translations are not within a distance of $O(1/\log T)$ of the critical axis, but instead are up to a distance of $O(1)$ away. In this case great deal of effort has been put into not only producing asymptotic formulas, but extracting all relevant lower order terms \cite{CoFaZi}, which have many interesting implications \cite{CoSn}. (We have called Conjecture \ref{RatiosConj} a `Local Ratios Conjecture' to distinguish it from this expanded set of conjectures.) Indeed, it is worth noting at this point that the formula in \eqref{EqRatios} is not the usual way to write the ratio conjecture; instead one usually insists that $\Re\, \beta_l, \beta_\ell' >0$ and conjectures that
\begin{equation}
\label{EqRatios2}
\lim_{T\rightarrow\infty} \frac{1}{T} \int_T^{2T} \prod_{l=1}^m \frac{\zeta\Big(\frac{1}{2} + \frac{\alpha_l}{\log T} + it\Big)}{\zeta\Big(\frac{1}{2}+ \frac{\beta_l}{\log T} + it\Big)} \prod_{\ell=1}^{m'}\frac{\zeta\Big(\frac{1}{2} + \frac{\alpha_\ell'}{\log T} + it\Big)}{\zeta\Big(\frac{1}{2}+ \frac{\beta_\ell'}{\log T} + it\Big)}\,dt
\end{equation}
is predicted accurately to first order by a random matrix analogue. The expression for this limit is somewhat more complicated to write down than the formula on the right hand side of \eqref{EqRatios} (see for instance \cite{CoFoSn,CoFaZi2, BuGa}). Nonetheless, in spite of the simplicity of \eqref{EqRatios}, it is not clear whether there is any way to write down the more precise lower-order Ratio Conjectures in a way reminiscent of it. It would still be interesting to see if such a combinatorial formalism can be found.

In any case, an asymptotic formula for the left hand side of \eqref{EqRatios} implies an asymptotic formula for the left hand side of \eqref{EqRatios2}, and vice-versa. This may be seen most easily by applying the zeta function's functional equation. We will have nothing to say about lower order terms however.

Similarly to Farmer's papers above, some previous work has studied the connections of the GUE Conjecture to averages of the logarithmic derivative of the zeta function \cite{GoGoMo, FaGoLeLe, Ro}.

We note also the concurrent work \cite{ChNaNi}, which considers some similar questions to those we consider here, but replaces the zeta function with a probabilistic construction called the limiting characteristic polynomial.

\subsection{}
We turn to a quick conceptual sketch of some of our methods. Both the moment bound, Theorem \ref{ratio_bound}, and the conditional implication, Theorem \ref{corr_implies_ratios}, are critically dependent on the tail bound, Theorem \ref{quad_bound}. The strategy in each case is to write
\begin{equation}
\label{linstat_exp}
\frac{\zeta\Big(\frac{1}{2} + \frac{\alpha}{\log T} + it\Big)}{\zeta\Big(\frac{1}{2}+ \frac{\beta}{\log T} + it\Big)} = \exp\Big[\underbrace{\mathrm{Log}\, \zeta\Big(\frac{1}{2} + \frac{\alpha}{\log T} + it\Big) - \mathrm{Log}\,\zeta\Big(\frac{1}{2}+ \frac{\beta}{\log T} + it\Big)}_{:=\mathcal{L}_t}\Big],
\end{equation}
(ignoring for the moment all issues with branch cuts, which end up being minor). We show from the Hadamard product representation for the zeta function that $\mathcal{L}_t$ is `very close' to a linear statistic $\sum \eta\Big(\frac{\log T}{2\pi}(\gamma-t)\Big)$, for some function $\eta$ of quadratic decay. This is not literally true: $\mathcal{L}_t$, if written as a sum of zeros, must contain an extra term in the summand that decays very slowly. This term \emph{does not} decay quadratically -- in fact its sum converges only because of the symmetry of zeros -- but it may be shown that on average this extra term does not much affect the size of $\mathcal{L}_t$. (This step is not trivial, but will be the content of Theorem \ref{J_bound} and Lemma \ref{W_markov}.)

Thus it is that we see that we can approximate the ratio \eqref{linstat_exp} by the exponential of a linear statistic of zeros. It is just these linear statistics whose size we have controlled in our tail bound, Theorem \ref{quad_bound}, and it is in this way that the moment bound Theorem \ref{ratio_bound} is proved. For the implication in Theorem \ref{corr_implies_ratios}, on the other hand, we note that we are able to asymptotically control the moments of such linear statistics by using the GUE Conjecture and a standard combinatorial procedure. This asymptotic control on the moments of linear statistics is not ipso facto enough to pass to the Local Ratios Conjecture however. It is not the case, that is, that Theorem \ref{corr_implies_ratios} is just a matter of combinatorial manipulation in random matrix theory. 

For instance, instead of $\mathcal{L}_t$, consider the random variables $X_n$ which take the value $0$ with probability $1-e^{-n}$ and $n^2$ with probability $e^{-n}$. Then $X_n$ tends to $0$ both in distribution and in the sense of moments: for any fixed $k\geq 0$,
$$
\bb{E}\, X_n^k \rightarrow 0.
$$
Yet
$$
\bb{E}\, e^{X_n} = (1-e^{-n}) + e^{n^2-n} \rightarrow\infty,
$$
so it is not true $\bb{E}\, e^{X_n} \sim \bb{E}\, e^0$.

This sort of a pathology is eliminated by the tail bound of Theorem \ref{quad_bound} and related bounds, and it is this control that is necessary to show that the average of ratios in \eqref{EqRatios} converges to a random matrix limit on the GUE Conjecture.

Our proof of Theorem \ref{quad_bound} is not long provided certain computational lemmas are taken on faith, so we will not sketch it here. We mention only that our proof depends on an application of Markov's inequality and a smoothing trick. It is, in this sense, an application of Soundararajan's method \cite{So} for bounding the moments of $\zeta(1/2+it)$ (see also Harper's refinement \cite{Ha}), used also his aforementioned work in \cite{So2}. 

Finally, we note that in the case that $\Re \alpha \leq \Re \beta$ and $\Im \alpha = \Im \beta$, there is an easier proof of the bound in Theorem \ref{ratio_bound}. In this case one has for all $t, T \geq 2$ a pointwise bound
$$
\zeta\Big(\frac{1}{2} + \frac{\alpha}{\log T} + it\Big) \ll_{\alpha,\beta} \zeta\Big(\frac{1}{2}+ \frac{\beta}{\log T} + it\Big).
$$
This is a consequence of Lemma 1 of \cite{Ra}. Nonetheless, such an inequality does not hold for other ranges of $\alpha$ and $\beta$, and Theorem \ref{ratio_bound} cannot in general be reduced to a pointwise estimate of this sort.

\textbf{Notation:} We follow standard conventions of analytic number theory, so that the notations $f(x) \ll g(x)$ and $f(x) = O(g(x))$ are interchangeable, with both meaning that $|f(x)|\leq C g(x)$ for all $x$, for a constant $C$. $f(x)\ll_A g(x)$ and $f(x) = O_A(g(x))$ both mean the constant $C$ may depend on $A$. The Fourier transform of a function $f$ is defined by $\hat{f}(\xi) := \int e^{-i2\pi x \xi} f(x)\,dx$. 

\emph{In what follows we will assume the Riemann hypothesis, without further statement of this assumption in Theorems, Lemmas, etc.}

\subsection{}
\textbf{Acknowledgments:} I thank Sandro Bettin, Alexei Borodin, Reda Chhaibi, Brian Conrey, Chris Hughes, Jon Keating, and Kurt Johansson for informative and encouraging discussions related to this work, and the anonymous referee for a careful reading and helpful suggestions.

\section{Bounding counts of zeros: a proof of Theorem \ref{quad_bound} and related bounds}
\label{2}

\subsection{}
As in many studies of the zeros of the zeta function, a principal tool is the explicit formula, due in stages to Riemann, Guinand, and Weil \cite{Ri, Gu, We}, relating the distribution of zeros to primes. A proof may be found in, for instance, \cite[pp. 410-416]{MoVa} or \cite[pp. 108-109]{IwKo}.

\begin{theorem}[The explicit formula]
\label{explicit}
For a compactly supported function $g$, piecewise continuous with finitely many discontinuities, such that $g(x) = \tfrac{1}{2}(g(x^-)+g(x^+))$ for all $x$ and $g(0) = \tfrac{1}{2}(g(x) + g(-x)) + O(|x|)$, we have,
$$
\lim_{V\rightarrow\infty} \sum_{|\gamma|< V}\hat{g}\Big(\frac{\gamma}{2\pi}\Big) - \int_{-V}^V \hat{g}\Big(\frac{\xi}{2\pi}\Big) \frac{\Omega(\xi)}{2\pi}\,d\xi = \int_{-\infty}^\infty (g(x) + g(-x)) e^{-x/2} d\big(e^x - \psi(e^x)\big),
$$
where
$$
\psi(x):= \sum_{n\leq x} \Lambda(n),
$$
with $\Lambda$ the von Mangoldt function, and 
$$
\Omega(\xi) := \frac{1}{2}\frac{\Gamma'}{\Gamma}\Big(\frac{1}{4}+i\frac{\xi}{2}\Big) +\frac{1}{2}\frac{\Gamma'}{\Gamma}\Big(\frac{1}{4}-i\frac{\xi}{2}\Big) - \log \pi.
$$
\end{theorem}
Using Stirling's formula for the digamma function \cite[Cor. 1.4.5]{AnAsRo}, one may verify that,
\begin{equation}
\label{Stirlings}
\frac{\Omega(\xi)}{2\pi} = \frac{\log\big((|\xi|+2)/2\pi\big)}{2\pi} + O\Big(\frac{1}{|\xi|+2}\Big).
\end{equation}
This term in the explicit formula therefore corresponds to an approximation of the density of zeros near height $\xi$. On the other hand,
$$
\int_{-\infty}^\infty g(x) e^{-x/2}d\big(e^x - \psi(e^x)\big) = \int_0^\infty \frac{g(\log t)}{\sqrt{t}}\,dt - \sum_{n=1}^\infty \frac{g(\log n)}{\sqrt{n}}\Lambda(n),
$$
and here the term $\int g(\log t)/\sqrt{t}\,dt$ serves as an approximation to $\sum g(\log n)\Lambda(n)/\sqrt{n}$.

Motivated by the explicit formula, we adopt the following notation, for a function $\eta$ of quadratic decay:
$$
\langle \eta, \Z \rangle = \langle \eta, \Z_T(t) \rangle := \sum_\gamma\eta\Big(\frac{\log T}{2\pi}(\gamma-t)\Big),
$$
$$
\langle \eta, \Z^{o} \rangle = \langle \eta, \Z_T^{o}(t) \rangle := \int_{-\infty}^\infty \eta\Big(\frac{\log T}{2\pi}(\xi-t)\Big) \frac{\Omega(\xi)}{2\pi}\,d\xi,
$$
$$
\langle \eta, \widetilde{\Z} \rangle = \langle \eta, \widetilde{\Z}_T(t) \rangle := \langle \eta, \Z \rangle - \langle \eta, \Z^{o} \rangle.
$$

Note that there is no question about the convergence of the sums or integrals in these definitions. We will later generalize this notation slightly, but we need not worry about this generalization for the moment. Note that for typographical reasons we will sometimes write $\Z$ or $\Z_T$ in place of $\Z_T(t)$. Unless otherwise indicated, $\Z = \Z_T = \Z_T(t)$, and likewise for $\Z^o$ and $\widetilde{\Z}$. 

We will see that the quantity $\langle \eta, \widetilde{\Z} \rangle$ and therefore $\langle \eta, \Z \rangle$ is approximated by a Dirichlet polynomial of length depending on the support of $\hat{\eta}$. It is in this way that we will control these quantities.

\subsection{}
Let $B_0$ be an absolute constant to be defined shortly. We define the function
\begin{equation}
\label{G_def}
G(\xi) := B_0\Big[\Big(\frac{\sin \pi (\xi + 1/4)}{\pi(\xi+1/4)}\Big)^2 + \Big(\frac{\sin \pi (\xi - 1/4)}{\pi (\xi - 1/4)}\Big)^2\Big],
\end{equation}
with Fourier transform,
\begin{equation}
\label{G_hat}
\hat{G}(x) = B_0(1-|x|)_+ \big(e^{i\pi x/2} + e^{-i\pi x/2}\big),
\end{equation}
where $B_0$ is an absolute constant chosen so that
\begin{equation}
\label{G_dom}
Q(\xi) \leq G(\xi)\,\quad \forall \xi\in \bb{R}.
\end{equation}
(In fact, $B_0$ may be chosen to be $2\pi^2$, but we only need to know such a constant exists, which is apparent from examining $G(\xi)/Q(\xi)$.) There is nothing very special about this test function $G$; we have chosen it to satisfy \eqref{G_dom} and
\begin{equation}
\label{G_supp}
\supp \hat{G} \subseteq [-1,1].
\end{equation}
As a consequence of \eqref{G_dom}, writing $G_k(\xi):= G(\xi/k)$, we see that for all $k\geq 1$,
\begin{equation}
\label{G_bound}
Q(\xi) \leq G_k(\xi),\quad \forall x\in \mathbb{R}.
\end{equation}
Moreover,
\begin{equation}
\label{Ghat_support}
\supp \hat{G}_k \subseteq [-1/k,1/k],\quad\,\textrm{with}\,\quad |\hat{G}_k(x)| \leq 2 B_0\, k (1-|kx|)_+.
\end{equation}

To make for a cleaner presentation, we work with notation from elementary probability, letting $t$ be a random variable uniformly distributed on the interval $[T,2T]$. The tail bound Theorem \ref{quad_bound} then becomes the claim that uniformly for $x\geq 2$ and $T\geq 1$,
$$
\bb{P}(\langle Q, \Z \rangle \geq x) \leq e^{-C x \log x}.
$$

The reason we have defined $G_k$ is that the size of $\langle Q, \Z \rangle$ can be controlled by $\langle G_k, \Z \rangle$, and that this in turn can be controlled by $\langle G_k, \Z^{o} \rangle$ and $\langle G_k, \widetilde{\Z}\rangle$. It is easy to control $\langle G_k, \Z^{o} \rangle$, since the measure defining this quantity is very regular. On the other hand $\langle G_k, \widetilde{\Z}\rangle$ can be well-controlled up to the $k$-th moment, with $\langle G_k, \widetilde{\Z}\rangle$ in general not being much larger then $\langle G, \widetilde{\Z} \rangle$. More exactly, we prove the following estimates.

\begin{lem}
\label{bulk}
For an absolute constant $B_1$, uniformly for $T\geq 2$,
$$
\langle Q, \Z_T(t) \rangle \leq B_1 \log T,\quad \forall t\in[T,2T].
$$
\end{lem}

\begin{lem}
\label{moments}
For an absolute constant $B_2,$ uniformly for $T\geq 2$ and $2\ell \leq k$, we have
$$
\bb{E} |\langle G_k, \widetilde{\Z} \rangle |^{2\ell} \leq (B_2\,\ell)^\ell.
$$
\end{lem}

\begin{lem}
\label{pointwise_regular}
For an absolute constant $B_3$, uniformly for $T\geq 2$ and $k \leq \sqrt{T}$,
$$
\langle G_k, \Z^o_T(t)\rangle  \leq B_3\, k,\quad \forall t\in [T,2T].
$$ 
\end{lem}

The first gives an extremely course upper bound for the number of zeros that may be counted by the test function $Q$, the second controls the moments of $\langle G_k, \widetilde{\Z}\rangle$ as described above, and the third controls the regular approximation $\langle G_k, \Z^o\rangle$ to the count of zeros by $G_k$.

These lemmas have standard proofs that we turn to at the end of this section -- the most nontrivial is Lemma \ref{moments} and is proved by approximating $\langle G_k, \widetilde{\Z} \rangle$ by a Dirichlet polynomial -- but before doing so, we show that with these computational estimates in hand, Theorem \ref{quad_bound} (our tail bound for zeros) follows quickly. 

\begin{proof}[Proof of Theorem \ref{quad_bound}]
Note first that in the case that $x > B_ 1 \log T$, Lemma \ref{bulk} implies that
$$
\bb{P}(\langle Q, \Z \rangle \geq x) = 0.
$$

We may therefore assume $x \leq B_1 \log T$. Lemma \ref{moments} allows us to see from Markov's inequality that for even integers $k$ and positive $y$,
\begin{equation}
\label{markov}
\bb{P}( \langle G_k, \widetilde{\Z} \rangle \geq y) \leq \frac{1}{y^k} \bb{E} |\langle G_k, \widetilde{\Z}\rangle |^k \ll \frac{(B_2\,k)^{k/2}}{y^k}.
\end{equation}
Yet
\begin{equation}
\label{positivity_prep}
\langle Q, \Z \rangle \leq \langle G_k, \Z \rangle = \langle G_k, \widetilde{\Z} \rangle + \langle G_k, \Z^o\rangle.
\end{equation}
Thus,
\begin{align}
\label{positivity}
\notag \bb{P}(\langle Q, \Z \rangle \geq x) &\leq \bb{P}(\langle G_k, \widetilde{\Z} \rangle + \langle G_k, \Z^o\rangle \geq x) \\
&\leq \bb{P}\big(\langle G_k, \widetilde{\Z} \rangle \geq x - B_3\, k \big),
\end{align}
for all even $k \leq \sqrt{T}$, with the last line following from Lemma \ref{pointwise_regular}. With no loss of generality, we may assume $x\geq 4 B_3$, and consider $k$ defined to be the positive even integer satisfying
$$
\frac{x}{2B_3} - 2 < k \leq \frac{x}{2B_3}
$$
so that in particular 
$$
x - B_3 k \geq x/2.
$$
As long as $T$ is large enough that $B_1/2B_2 \, \log T \leq \sqrt{T}$, then certainly $k\leq \sqrt{T}$ (since we are considering the case $x \leq B_1\log T$). Thus from \eqref{markov} and \eqref{positivity},
\begin{equation}
\label{quad_bound_eq}
\bb{P}(\langle Q, Z \rangle \geq x) \leq \frac{(B_2 \cdot x/2B_3)^{x/4B_3}}{(x/2)^{x/2B_3-2}} \ll e^{-C x\log x},
\end{equation}
for an absolute constant $C$.\footnote{An argument with more bookkeeping, though still one which makes no attempt at optimization, shows that one may take any constant $C < 1/16\pi^2$, for instance.} 

In remains to verify our claim in the case in which $T$ is small enough that $B_1/2B_2 \log T > \sqrt{T}$. But this bounded range of $T$ can at most alter the implicit constant in \eqref{quad_bound_eq}.
\end{proof}

\textbf{Remark:} There is a slightly different approach to this theorem which some readers may prefer. Instead of the inequality \eqref{positivity_prep}, we may make use of a mollification formula of Selberg \cite[Th. 1]{Se1}, which approximates the classical function $S(t)$ by a Dirichlet polynomial with error terms whose size depends on the length of the Dirichlet polynomial. One may then compute moments of, say, $S(t+1/\log T) - S(t)$ in the same way we have here, with the Dirichlet polynomial replacing the quantity $\langle G_k, \widetilde{\Z} \rangle$. 

Indeed, to reflect on our approach, in the lemmas, it has been to show the following:
\begin{equation}
\label{remark1}
N\Big(t+\frac{1}{\log T}\Big) - N(t) \ll \langle Q, \Z_T(t) \rangle \ll \langle G_k, \widetilde{\Z}_T(t) \rangle + k,
\end{equation}
with $k \geq 1$. By the explicit formula, we will reduce $\langle G_k, \widetilde{\Z}\rangle$ to a Dirichlet polynomial in the proof of Lemma \ref{moments} below in order to compute its moments. In slightly more traditional notation, with such a Dirichlet polynomial already put in place of $\langle G_k, \widetilde{\Z} \rangle$, \eqref{remark1} could be rewritten
\begin{equation}
\label{remark2}
N\Big(t+\frac{1}{\log T}\Big) - N(t) \ll \frac{1}{\log x} \Im \sum_{p \leq x} \Big(1-\frac{\log p}{\log x}\Big)\frac{\log p}{p^{1/2+it}} + \frac{\log T}{\log x},
\end{equation}
for $t\in [T,2T]$ and all $2 \leq x \leq T$ (and $x$ related to $k$ above by $k = \tfrac{\log T}{\log x}$). For such Dirichlet polynomials, we will be able to bound $k$-th moments, and thereby control how frequently $N(t+1/\log T)-N(t)$ can be large.

We have taken the route and notation that we have because we will make use of the same formalism elsewhere in this paper; we apply it to other estimates for zeta zeros below, and it applies almost without change to study the eigenvalues of the unitary group, for instance.

\textbf{Remark:} Without the Riemann hypothesis, the ordinates $\gamma$ needn't be real, and the relationship $\langle Q, \Z \rangle \leq \langle G_k, \Z \rangle$ ceases to hold; the same is true of \eqref{remark2}. On the other hand, Selberg \cite[Th. 2]{Se2} also proves an uncondtional variant of his approximation for $S(t)$, and this has been used by Fujii \cite[p. 245]{Fu} to compute moment bounds for $S(t+1/\log T) - S(t)$ unconditionally. Bounds that can be obtained unconditionally in this way are slightly worse than what we have derived assuming RH. Unconditionally, using the technique, one can prove $
\frac{1}{T}\,\mathrm{meas}\big\{ t\in [T,2T]:\, N(t+1/\log T)-N(t) \geq x\big\} \ll e^{-c x}
$, where $c$ is an absolute constant, but seemingly no better. It would be interesting to see if this could be improved.

\textbf{Remark:} Probably the tail bounds in Theorem \ref{quad_bound} and Corollary \ref{interval_bound}, while sufficient for our purposes, are not optimal. The bounds here would correspond to the `right answer' were the zeros were modeled by a Poisson process, but since zeros of the zeta function tend to repel each other one might guess that the counts are sub-gaussian in Theorem \ref{quad_bound} and Corollary \ref{interval_bound}. Such an estimate is true for eigenvalues of the unitary group -- see \eqref{gaussian_quad_bound_unitary} below -- but for zeta zeros seemingly this is a harder statement to prove.

\subsection{}
There is another result similar to Theorem \ref{quad_bound} that we will require, but which is somewhat more technical in its statement and proof. We generalize the notation $\langle \eta, \widetilde{\Z}\rangle$ to a wider class of functions than it was applied to before. In particular, we let
$$
\langle \eta, \widetilde{\Z} \rangle = \langle \eta, \widetilde{\Z}_T(t) \rangle := \lim_{V\rightarrow \infty} \sum_{|\gamma| < V} \eta\Big(\frac{\log T}{2\pi}(\gamma-t)\Big) - \int_{-V}^V \eta\Big(\frac{\log T}{2\pi}(\xi-t)\Big) \frac{\Omega(\xi)}{2\pi}\,d\xi,
$$
where $\eta, T,$ and $t$ are such that the limit exists. This is consistent with our previous use of this notation. Likewise, when the limit exists,
$$
\langle \eta, \Z \rangle = \langle \eta, \Z_T(t) \rangle:= \lim_{V\rightarrow \infty} \sum_{|\gamma| < V} \eta\Big(\frac{\log T}{2\pi}(\gamma-t)\Big),
$$
$$
\langle \eta, \Z^o \rangle = \langle \eta, \Z^o_T(t) \rangle:= \lim_{V\rightarrow \infty}\int_{-V}^V \eta\Big(\frac{\log T}{2\pi}(\xi-t)\Big) \frac{\Omega(\xi)}{2\pi}\,d\xi.
$$
By the explicit formula, it may be verified that $\langle \eta, \widetilde{\Z} \rangle$ exists whenever $\eta(\xi) = \hat{f}(\xi)$, for a function $f$ that is (i) compactly supported, (ii) piecewise continuous with finitely many discontinuities, (iii) satisfying $f(x) = \tfrac{1}{2}(f(x^+) + f(x^-))$, and (iv) with $f$ odd. A more specific example of such a limit existing where the sums and integral do not absolutely converge is furnished by the function
\begin{equation}
\label{J_def}
J(\xi):= \frac{2\pi \xi}{1+(2\pi \xi)^2}.
\end{equation}
In this case, $J(\xi) = \hat{f}(\xi)$, for the function 
\begin{equation}
\label{J_fourier}
f(x):= -\sgn(x) e^{-|x|}/2i,
\end{equation} 
so one may see by the above discussion that $\langle J, \widetilde{Z}_T(t)\rangle$ is well defined for all $T$ and $t$. Alternatively, one may see rather more simply that the limit defining $\langle J, \widetilde{Z}\rangle$ converges by exploiting the symmetry of the zeros $\gamma$ and the function $\Omega$. Indeed, let us verify this (and prove a little more) for $\langle J, \Z^o\rangle$, in a lemma we will need later.

\begin{lem}
\label{J_regular}
Uniformly for $T\geq 2$,
$$
\langle J, \Z^o_T(t)\rangle = O(1/\log T),\quad \forall t\in[T,2T].
$$
\end{lem}

\begin{proof}
By the symmetry of $J$,
\begin{align*}
\langle J, \Z^o_T(t)\rangle =& \lim_{V\rightarrow\infty} \int_0^V J(y) \Big[\Omega\Big(t+ \frac{2\pi y}{\log T}\Big) - \Omega\Big(t-\frac{2\pi y}{\log T}\Big)\Big]\,\frac{dy}{\log T} \\
\ll& \frac{1}{\log T}\int_0^\infty J(y)\Big[\min\Big(\frac{y^2}{t^2\log^2 T}, \frac{t^2 \log^2 T}{y^2}\Big) + O\Big(\frac{1}{\big|t-\tfrac{2\pi y}{\log T}\big|+2}\Big)\Big]\,dy,
\end{align*}
where in the second step in approximating $\Omega$, we have used Stirling's formula \eqref{Stirlings} and then simple Taylor series estimates for the logarithm function. (Note that in the first line the integrand is positive, so the integral converges absolutely or not at all.) It is now slightly tedious but straightforward to verify that the integral is $O(1)$ and therefore the entire expression is $O(1/\log T)$.
\end{proof}

The analogue of Theorem \ref{quad_bound}, our earlier tail bound, that we require is the following.

\begin{theorem} (Tail bound for signed counts)
\label{J_bound}
For all $x \geq 2$ and all $T \geq 2$,
$$
\bb{P}(\, |\langle J, \widetilde{\Z}\rangle| \geq x\,) \ll e^{-C x\log x},
$$
where the constant $C$ and the implicit constant are absolute.
\end{theorem}

Applying Lemma \ref{J_regular} here, we see likewise:

\begin{cor}[Tail bound for signed counts]
\label{JZ_bound}
For all $x \geq 2$ and all $T \geq 2$,
$$
\bb{P}(\, |\langle J, \Z \rangle| \geq x\,) \ll e^{-C x\log x},
$$
where the constant $C$ and the implicit constant are absolute.
\end{cor}

Our proof of Theorem \ref{J_bound} is similar to the proof of Theorem \ref{quad_bound}. Again we require a series of lemmas, to be proved later.

\begin{lem}
\label{J_bulk}
For an absolute constant $B_1'$, uniformly for $T\geq 2$,
$$
|\langle J, \widetilde{Z}_T(t)\rangle | \leq B_1' \log T, \quad \forall t\in [T,2T].
$$
\end{lem}

For the next two lemmas we define
\begin{equation}
\label{W_def}
W^{(\epsilon)}(x):= \frac{\sgn(x)}{-2i} e^{-|x|} (1-|x|/\epsilon)_+.
\end{equation}

We have defined $W^{(\epsilon)}$ so that $(W^{(1/k)})\,\hat{}$, for $k\geq 1$, plays the role of something like a smooth approximation to the function
\begin{equation}
\label{xi_tails}
J(\xi)\mathbf{1}_{|\xi|\geq k}.
\end{equation}
More exactly, a computation reveals that,
\begin{equation}
\label{W_hat}
(W^{(1/k)})\,\hat{}\,(z)= J(z) + \frac{k}{2i} \Big( \underbrace{\frac{1-\exp\big(\tfrac{1+i2\pi z}{k}\big)}{(1+i2\pi z)^2}}_{K(z),\,\textrm{say}} - \underbrace{\frac{1-\exp\big(\tfrac{1-i2\pi z}{k}\big)}{(1-i2\pi z)^2}}_{K(-z)}\Big).
\end{equation}
(We have written $z$ instead of $\xi$ here, because we will later need this expression for complex values of $z$ as well.) We will see from a Taylor expansion, $(W^{(1/k)})\,\hat{}\,(\xi)$ is small when $|\xi| \leq k$, and the terms $K(\xi)$ may be thought of as an error term when $|\xi|$ is large. A more exact statement of this is as follows: 

\begin{lem}
\label{J_approx}
For all $k\geq 1$,
$$
|J(\xi) - (W^{(1/k)})\,\hat{}\,(\xi)| \leq A\, G_k(\xi), \quad \forall \xi \in \bb{R},
$$
where $A$ is an absolute constant.
\end{lem}

We also have the moments of $(W^{(1/k)})\,\hat{}$ are very small when $k$ is large.

\begin{lem}
\label{J_moments}
For an absolute constant $B_2'$, uniformly for $T\geq 2$ and $2\ell \leq k \leq \sqrt{T}$, we have 
$$
\bb{E} |\langle (W^{(1/k)})\,\hat{}\,, \widetilde{\Z}\rangle|^{2\ell} \leq (B_2' \ell)^\ell k^{-2\ell}.
$$
\end{lem}

As before, we momentarily delay the proof of these lemmas. Assuming them, we see that a proof of Theorem \ref{J_bound}, the tail bound for oscillatory counts, proceeds in the same manner as that of Theorem \ref{quad_bound}, the tail bound for quadratically decaying counts.

\begin{proof}[Proof of Theorem \ref{J_bound}]
If $x > B_1' \log T$, then by Lemma \ref{J_bulk},
$$
\bb{P}(|\langle J, \widetilde{\Z}\rangle| \geq x) = 0.
$$
So as before we may treat the case that $x \leq B_1' \log T$. By applying Lemma \ref{J_approx}, for all $k\geq 1$,
$$
\langle J, \widetilde{\Z}\rangle = \langle (W^{(1/k)})\,\hat{}\,,\widetilde{\Z}\rangle + O(|\langle G_k, \widetilde{\Z}\rangle |) + O(\langle G_k, \Z^o\rangle ),
$$
where the implicit constant in the first error term may be taken as $A$, and the implicit constant in the second $2A$. As long as $k\leq \sqrt{T}$, Lemma \ref{pointwise_regular} allows us to bound the second of these error terms: $\langle G_k, \Z^o\rangle \leq B_3\,k.$ Hence using a union bound,
$$
\bb{P}(\langle J, \widetilde{\Z}\rangle | \geq x) \leq \bb{P}\big(|\langle (W^{(1/k)})\,\hat{}\,,\widetilde{\Z}\rangle| \geq x/2\big) + \bb{P}\big(A|\langle G_k, \widetilde{\Z} \rangle| + AB_3\,k \geq x/2\big).
$$
A choice of $k$ and bound for both probabilities then proceeds as in the proof of Theorem \ref{quad_bound}, replacing Lemma \ref{moments} by Lemma \ref{J_moments} to bound the first of these terms.
\end{proof}

There is one last result of this sort that we will use below.

\begin{lem}[Bound on $1/\xi$ stubs]
\label{W_markov}
For any $\epsilon\geq 0$ and $k\geq 1/\epsilon^2$, for $T = T(\epsilon)$ sufficiently large,
$$
\bb{P}(|\langle (W^{(1/k)})\,\hat{}\,,\Z_T\rangle| \geq \epsilon) \ll \epsilon^2.
$$
\end{lem}

By the approximation \eqref{xi_tails}, this roughly corresponds to a statement that when $k$ is large $$\sum_{\Big|\frac{\log T}{2\pi}(\gamma-t)\Big| \geq k} \frac{1}{\tfrac{\log T}{2\pi}(\gamma-t)}$$ is typically very small. This is a result, in part, of cancellation between the two `sides' of the sum. 

Finally, at the end of this section, we explain how it is that Theorem \ref{J_bound} implies Theorem \ref{log_bound}.

\subsection{}
We finally turn to proofs of the lemmas above.

\begin{proof}[Proof of Lemma \ref{bulk}]
We recall the estimate (see \cite[Cor. 14.3]{MoVa}),
$$
N(t+1)-N(t) \ll \log(|t|+2),\quad \forall t\in\bb{R}.
$$
By inspection, it is easy to verify that $\log(|u+v|+2) \ll \log(|u|+2) + \log(|v|+2)$ for all $u,v\in\bb{R}$. 

Now note that for $t\in[T,2T]$,
\begin{align*}
\langle Q, \Z_T(t) \rangle &\ll \sum_{k=-\infty}^\infty [N(t+k+1)-N(t+k)]\cdot \frac{1}{1+ k^2 \log^2 T} \\
&\ll \log(T+2)\sum_{k=-\infty}^\infty \frac{1}{1+k^2\log^2T} + \sum_{k=-\infty}^\infty \frac{1}{1+k^2 \log^2 T} \log(|k|+2) \\
&\ll \log(T).
\end{align*}
\end{proof}

\begin{proof}[Proof of Lemma \ref{moments}]
This is a more or less standard computation of moments. However, some added care is necessary since an estimate is required that is uniform as moments vary. We note that from the explicit formula,
\begin{align*}
\langle G_k, \widetilde{\Z}_T(t) \rangle =& \frac{k}{\log T} \int_{-\infty}^\infty \hat{G}\Big(\frac{kx}{\log T}\Big)(e^{-ixt} + e^{ixt}) e^{-x/2}d\big(e^x - \psi(e^x)\big) \\
=&  \mathcal{I} - 2\Re \frac{k}{\log T} \underbrace{\sum_{r\geq 1} \sum_p \hat{G}\Big(\frac{kr}{\log T}\log p\Big) \frac{\log p}{p^{r(1/2-it)}}}_{\sum_r s_r,\; \textrm{say}},
\end{align*}
where
$$
\mathcal{I}:= G\Big(\frac{\log T}{2\pi k}(i/2 - t)\Big) + G\Big(\frac{\log T}{2\pi k}(i/2 + t)\Big).
$$
Here $G(x+iy)$ is, of course, the analytic continuation of the function defined before in \eqref{G_def}. One may check that
$$
G\Big(\frac{\log T}{2\pi k}(i/2 \pm t)\Big) \ll \frac{\exp(\log T/4k)}{t^2} \ll \frac{1}{T^{7/4}} \ll 1,
$$
for $k\geq 1$ and $t\in [T,2T]$. 

Hence from H\"{o}lder's inequality,
\begin{equation}
\label{holder}
\Big(\bb{E} \langle G_k, \widetilde{\Z} \rangle^{2\ell} \Big)^{1/2\ell} \ll 1 + \frac{k}{\log T}\sum_{r\geq 1} \Big(\bb{E} |s_r|^{2\ell}\Big)^{1/2\ell}.
\end{equation}

Because $\supp \hat{G}\subseteq [-1,1]$, a standard argument dating back to Selberg (see \cite[Lem. 3]{So}, for a modern treatment that applies directly) reveals\footnote{In fact, the argument shows that up to twice our range of $\ell$ may be admitted.} that for $2\ell \leq k$ 
$$
\bb{E} |s_r|^{2\ell} \ll \ell!\cdot \bigg(\sum_p \frac{\log^2 p}{p^r} \hat{G}\Big(\frac{k r}{\log T}\log p\Big)^2\bigg)^\ell.
$$
By the support of $\hat{G}$, this quantity is null for all $r \geq 1$, when $k > \log T/\log 2$. 

In the case that $k \leq \log T/\log 2$ we need a little more work. When $r=1$,
$$
\sum_p \frac{\log^2 p}{p} \hat{G}\Big(\frac{kr}{\log T}\log p\Big)^2 \ll  \sum_{p \leq T^{1/k}} \frac{\log^2 p}{p} \ll \Big(\frac{\log T}{k}\Big)^2,
$$
by Chebyshev (see \cite[Ch. 2.2]{MoVa}). When $r\geq 2$,
$$
\sum_p \frac{\log^2 p}{p^r} \hat{G}\Big(\frac{kr}{\log T}\log p\Big)^2 \ll \int_2^\infty \frac{\log^2 t}{t^r}\,dt \ll \frac{1}{2^r}.
$$
Returning to \eqref{holder}, we see that
\begin{align*}
\Big(\bb{E} \langle G_k, \widetilde{\Z} \rangle^{2\ell} \Big)^{1/2\ell} &\ll 1+ \frac{k}{\log T} (\ell!)^{1/2\ell} \Big(\frac{\log T}{k} + \sum_{r\geq 2} \frac{1}{2^{r/2}}\Big)\\
&\ll \ell^{1/2},
\end{align*}
as $k\leq \log T/\log 2$. We have used Stirling's formula \cite[Th. 1.4.1]{AnAsRo} to bound the factorial. Exponentiating by $2\ell$ gives the lemma.
\end{proof}

\begin{proof}[Proof of Lemma \ref{pointwise_regular}]
We have
\begin{align*}
\langle G, \Z^o_T(t)\rangle =& \frac{1}{\log T} \int_{-\infty}^\infty G_k(y) \Omega\Big(t + \frac{2\pi y}{\log T}\Big) \, dy \\
=& \frac{1}{\log T}\bigg(\int_{|y|\leq T \log T} + \int_{|y| > T \log T}\bigg) G_k(y) \Omega\Big(t + \frac{2\pi y}{\log T}\Big) \, dy.
\end{align*}
By our application of Stirling's formula \eqref{Stirlings}, this quantity is
\begin{align*}
\ll& \frac{1}{\log T}\bigg(\int_{|y|\leq T\log T} G_k(y) \log T\, dy + \int_{|y| > T\log T} \frac{k^2}{y^2} \log y\, dy \bigg) \\
\ll& k + k^2/T,
\end{align*}
which yields the estimate.
\end{proof}

\begin{proof}[Proof of Lemma \ref{J_bulk}]
It is easy to verify for $\Re s > 1/2$ that
$$
\int_0^\infty e^{-sx} e^{-x/2} d(\psi(e^x)-e^x) = - \frac{\zeta'}{\zeta}\Big(\frac{1}{2}+s\Big)-\frac{1}{s-1/2}.
$$
On RH, by analytic continuation, this identity remains true for $\Re s > 0$. Making use of this identity, the Fourier transform expression \eqref{J_fourier}, and the explicit formula, one may thus verify that
\begin{align}
\label{J_zeta}
\notag \langle J, \widetilde{\Z}_T(t) \rangle =& \frac{1}{\log T} \Im \frac{\zeta'}{\zeta}\Big(\frac{1}{2}+\frac{1}{\log T} + it\Big) \\
\notag &+ \frac{1}{\log T}\Big( \frac{t}{\big(\tfrac{1}{2}-\tfrac{1}{\log T}\big)^2 + t^2} - \frac{t}{\big(\tfrac{1}{2}+\tfrac{1}{\log T}\big)^2 + t^2}\Big) \\
=& \frac{1}{\log T} \Im \frac{\zeta'}{\zeta}\Big(\frac{1}{2}+\frac{1}{\log T} + it\Big) + O\Big(\frac{1}{\log T}\Big).
\end{align}
From Lemma 12.1 of \cite{MoVa}, we see that for $t\in[T,2T]$,
\begin{align*}
\frac{\zeta'}{\zeta}\Big(\frac{1}{2}+\frac{1}{\log T} + it\Big) =& O(1) + \sum_{|\gamma-t| \leq 1} \frac{1}{1/\log T - i(\gamma-t)} \\
=& O(1) + O(\log^2 T) \\
=& O(\log ^2 T),
\end{align*}
with the second to last line following from the fact that $N(t+1)-N(t) = O(\log(|t|+2)$. Combining this estimate with \eqref{J_zeta} yields the lemma.
\end{proof}

\begin{proof}[Proof of Lemma \ref{J_approx}]
A Taylor expansion of the exponential function in \eqref{W_hat} shows that for $|\xi|\leq k$ (throughout this proof, $\xi$ is real),
\begin{equation}
\label{W_bound_small}
(W^{(1/k)})\,\hat{}\,(\xi) \ll 1/k.
\end{equation}
In the same range, plainly $J(\xi) \ll 1$. Hence, for $|\xi| \leq k$,
$$
|(W^{(1/k)})\,\hat{}\,(\xi) - J(\xi) | \ll 1 \ll G_k(\xi).
$$
On the other hand, for $|\xi| > k$, by \eqref{W_hat},
\begin{equation}
\label{W_bound_large}
|(W^{(1/k)})\,\hat{}\,(\xi) - J(\xi) | \ll \frac{k}{\xi^2} \ll \frac{k^2}{\xi^2} \ll G_k(\xi).
\end{equation}
\end{proof}

\begin{proof}[Proof of Lemma \ref{J_moments}]
Our proof proceeds along the same lines as that of Lemma \ref{moments}. From the explicit formula,
$$
\langle (W^{(1/k)})\,\hat{}\,, \widetilde{\Z}_T(t)\rangle = \mathcal{I}' + \Im \frac{1}{\log T} \underbrace{\sum_{r\geq 1} \sum_p \frac{\log p}{p^{r(1/2 + 1/\log T -it)}}\Big(1- k \frac{r \log p}{\log T}\Big)_+ }_{\sum_r \sigma_r,\; \textrm{say}},
$$
where 
$$
\mathcal{I}':= (W^{(1/k)})\,\hat{}\,\Big(\frac{\log T}{2\pi}(i/2-t)\Big) + (W^{(1/k)})\,\hat{}\,\Big(\frac{\log T}{2\pi}(-i/2-t)\Big).
$$
To bound $\mathcal{I}'$, we recall \eqref{W_hat}. It is simple to verify that for $t\in[T,2T]$,
$$
J\Big(\frac{\log T}{2\pi}(\pm i/2-t)\Big) \ll \frac{1}{T\log T}\ll \frac{1}{k}
$$
in the range that $k \leq \sqrt{T}$. On the other hand, a bit more tediously,
$$
K\Big(\pm\frac{\log T}{2\pi}(\pm i/2-t)\Big) \ll \frac{k}{(1-\log T/2)^2+t^2}\exp\Big(\frac{\log T}{2k}\Big) \ll \frac{k}{T^{7/4}} \ll \frac{1}{k},
$$
again for $k \leq \sqrt{T}$. This shows that
$$
\mathcal{I}' \ll \frac{1}{k}.
$$
Thus, as in the H\"older inequality \eqref{holder} of the proof of Lemma \ref{moments},
\begin{equation}
\label{W_holder}
\Big(\bb{E} \langle (W^{(1/k)})\,\hat{}\,, \widetilde{\Z} \rangle^{2\ell} \Big)^{1/2\ell} \ll \frac{1}{k} + \frac{1}{\log T}\sum_{r\geq 1} \Big(\bb{E} |\sigma_r|^{2\ell}\Big)^{1/2\ell}.
\end{equation}
But also as in that proof, for $k > \log T/\log 2$,
$$
\sigma_r = 0, \quad \forall r\geq 1.
$$
Otherwise, for $k \leq \log T/\log 2$, the right hand side of \eqref{W_holder} is likewise bound by
\begin{align*}
\ll& \frac{1}{k} + \frac{1}{\log T}(\ell!)^{1/2\ell} \Big(\frac{\log T}{k} + \sum_{r\geq 2} \frac{1}{2^{r/2}}\Big) \\
\ll& \ell^{1/2}/k.
\end{align*}
This proves the lemma.
\end{proof}

\begin{proof}[Proof of Lemma \ref{W_markov}]
We begin by considering the case that $\epsilon > 1/2$. In this case, the lemma is tautological:
$$
\bb{P}\big(|\langle (W^{(1/k)})\,\hat{}\,,\Z \rangle | \geq \epsilon\big) \leq 1 \ll \epsilon^2.
$$

We may therefore suppose $\epsilon\in(0,1/2)$. From Lemma \ref{J_moments}, we see (noting that this condition on $\epsilon$ imposes $k\geq 2$),
\begin{align*}
\bb{P}\big(|\langle (W^{(1/k)})\,\hat{}\,,\widetilde{\Z} \rangle | \geq \epsilon\big) \leq & \frac{1}{\epsilon^2}\bb{E}\big(|\langle (W^{(1/k)})\,\hat{}\,,\widetilde{\Z} \rangle |^2\big) \\
\ll& \frac{1}{(\epsilon k)^2} \leq \epsilon^2.
\end{align*}

On the other hand, from \eqref{W_bound_small} and \eqref{W_bound_large}, for all $k\geq 1$, we have 
\begin{equation}
\label{W_ll_J}
(W^{(1/k)})\,\hat{}\,(\xi) \ll J(\xi),\quad \forall\xi\geq 0.
\end{equation}
Hence, using the symmetry of $(W^{(1/k)})\,\hat{}$ in the first line below,
\begin{align*}
\langle (W^{(1/k)})\,\hat{}\,, \Z^o_T(t)\rangle =& \lim_{V\rightarrow\infty} \int_0^V (W^{(1/k)})\,\hat{}\,(y) \Big[\Omega\Big(t+ \frac{2\pi y}{\log T}\Big) - \Omega\Big(t-\frac{2\pi y}{\log T}\Big)\Big]\,\frac{dy}{\log T} \\
\ll& \langle J, \Z^o_T(t)\rangle
\ll \frac{1}{\log T}.
\end{align*}
We are justified in applying the bound \eqref{W_ll_J} in passing to the second line because, as in the proof of Lemma \ref{J_regular}, $\Omega(t+ 2\pi y/\log T) - \Omega(t-2\pi y/\log T)\geq 0$ for all $y\geq 0$.

Thus for sufficiently large $T$ (such that $1/\log T$ is small in comparison to $\epsilon$), 
$$
\bb{P}\big(|\langle (W^{(1/k)})\,\hat{}\,,\Z \rangle | \geq \epsilon\big) \leq \bb{P}\big(|\langle (W^{(1/k)})\,\hat{}\,,\widetilde{\Z} \rangle | \geq \epsilon/2 \big)  \ll \epsilon^2,
$$
as claimed.
\end{proof}

\subsection{}
Finally, we turn to bounding the logarithmic derivative of the zeta function. Some computational details in the proof are left to the reader.

\begin{proof}[Proof of Theorem \ref{log_bound}]
We note that by much the same procedure as in \eqref{J_zeta}, we have
\begin{equation}
\label{log_d_zeta}
\frac{1}{\log T} \frac{\zeta'}{\zeta}\Big(\frac{1}{2}+\frac{\alpha}{\log T} + it\Big) = \langle I_\alpha, \widetilde{\Z}_T(t) \rangle + O_\alpha\Big(\frac{1}{\log T}\Big),
\end{equation}
for $t\in [T,2T]$, where
$$
I_\alpha(\xi):= \frac{1}{\alpha - i 2\pi \xi}.
$$
Because $I_\alpha(\xi) = J(\xi) + O_\alpha(Q(\xi))$, the claim follows directly from Theorems \ref{quad_bound} and \ref{J_bound}.
\end{proof}

\textbf{Remark:} An alternative approach to the identity \eqref{log_d_zeta} is to take as a starting point the classical formula \cite[Corollary 10.14]{MoVa},
$$
\frac{\zeta'}{\zeta}(s) = \frac{-1}{s-1} + \sum_\rho \Big(\frac{1}{s-\rho} + \frac{1}{\rho}\Big) - \frac{1}{2}\log(|t|+2) + O(1),
$$
where $s = \sigma+it$. This is less exact algebraically, but expresses the same idea.

\section{Ratio bounds}
\label{3}
With the bounds of Theorems \ref{quad_bound} and \ref{J_bound} in place, it is a simple matter to bound moments of ratios of the zeta function.

\begin{proof}[Proof of Theorem \ref{ratio_bound}]
In this proof we assume $\Re\,\beta \neq 0$ throughout. Using the Hadamard product representation for the zeta function \cite[Th. 10.12]{MoVa} and Stirling's formula for the Gamma function \cite[Cor. 1.4.3]{AnAsRo}, it is straightforward (though a little tedious) to verify that for fixed $\alpha, \beta \in \bb{C}$,
\begin{equation}
\label{hadamard}
\frac{\zeta\Big(\frac{1}{2} + \frac{\alpha}{\log T} + it\Big)}{\zeta\Big(\frac{1}{2}+ \frac{\beta}{\log T} + it\Big)} = (1+o(1)) e^{-(\alpha-\beta)/2} \lim_{V\rightarrow\infty} \prod_{|\gamma|\leq V} \frac{\frac{\alpha}{2\pi} - i \frac{\log T}{2\pi}(\gamma-t)}{\frac{\beta}{2\pi} - i \frac{\log T}{2\pi}(\gamma-t)},
\end{equation}
where because $\Re \beta \neq 0$ the product converges to a finite number (on RH). Here $o(1)$ is a quantity that tends to $0$ uniformly for $t\in[T,2T]$ as $T\rightarrow\infty$.

Using
$$
\mathrm{Log}(z):= \log|z| + i\mathrm{Arg}(z),
$$
with $\mathrm{Arg}(z) \in (-\pi, \pi]$ for all $z\in \bb{C}$, and defining
$$
L_{\alpha,\beta}(\xi):=\mathrm{Log}\Big(\frac{\frac{\alpha}{2\pi}-i\xi}{\frac{\beta}{2\pi}-i\xi}\Big),
$$
one sees that the expression \eqref{hadamard} is equal to
\begin{equation}
\label{ratio_linear_stat}
(1+o(1))e^{-(\alpha-\beta)/2} \exp\big(\langle L_{\alpha,\beta}\,, \Z_T(t) \rangle \big)
\end{equation}
(A little care must be taken, of course, whenever taking the logarithm of a complex number, but here, due to the exponential, no problems arise. One must check that the sum defining $\langle L_{\alpha,\beta}\,,\Z_T(t)\rangle$ converges, but this is straightforward using the symmetry of $\gamma$.)

Note that for $|\xi| > \max(2\pi|\alpha|, 2\pi|\beta|)$
$$
L_{\alpha,\beta}(\xi) = \mathrm{Log}\Big(\frac{1-\frac{\alpha}{i2\pi\xi}}{1-\frac{\beta}{i2\pi\xi}}\Big) = i(\alpha-\beta)J(\xi) + O_{\alpha,\beta}(Q(\xi)),
$$
while, as long as $\Re \beta \neq 0$, we have for $|\xi| < \max(2\pi|\alpha|, 2\pi|\beta|)$,
$$
\frac{\frac{\alpha}{2\pi}-i\xi}{\frac{\beta}{2\pi}-i\xi} = \exp\big[O_{\alpha, \beta}\big(Q(\xi)\big)\big],
$$
since for this region of $\xi$, the left hand side is bounded above, and the right hand side is bounded from below. Hence for all $\xi\in\bb{R}$,
\begin{equation}
\label{L_estimate}
L_{\alpha,\beta}(\xi) = i(\alpha-\beta) J(\xi) + O_{\alpha,\beta}(Q(\xi)).
\end{equation}
Thus for fixed $\alpha,\beta, m$, with $\Re\, \beta \neq 0$,
\begin{align*}
\Bigg| \frac{\zeta\Big(\frac{1}{2} + \frac{\alpha}{\log T} + it\Big)}{\zeta\Big(\frac{1}{2}+ \frac{\beta}{\log T} + it\Big)}\Bigg|^m =& (1+o(1))e^{-m(\alpha-\beta)/2}\exp\big(m\Re\, \langle L_{\alpha,\beta},\Z \rangle\big) \\
=& (1+o(1))e^{-m(\alpha-\beta)/2}\exp\big[O\big(\langle J, \Z\rangle \big) + O\big( \langle Q, \Z\rangle \big) \big].
\end{align*}
Now the theorem at hand follows from Theorem \ref{quad_bound} (our tail bound for zeros) and Corollary \ref{JZ_bound} (our tail bound for oscillatory counts).
\end{proof}

\textbf{Remark:} There is an alternative to the identity \eqref{ratio_linear_stat} that is more exact algebraically. Under RH, it may be seen (for instance, with \cite[Eq. (14.10.5)]{Ti} as a starting point) that for $\Re\, \alpha, \beta > 0$,
\begin{equation}
\label{hadamard_alternative}
\frac{\zeta\Big(\frac{1}{2} + \frac{\alpha}{\log T} + it\Big)}{\zeta\Big(\frac{1}{2}+ \frac{\beta}{\log T} + it\Big)} = \exp\big(\langle L_{\alpha,\beta},\widetilde{\Z}\rangle\big).
\end{equation}
To use this identity in the proof above to treat those values of $\alpha$ or $\beta$ with negative real part, the functional equation must be made use of.

\section{A random matrix interlude}
\label{4}

\subsection{}
In this section, we develop analogues for the unitary group of our tail bound for linear statistics (for zeta zeros this was Theorem \ref{quad_bound}), the determinantal evaluation of correlation functions (for the zeta zeros this was Conjecture \ref{GUEConj}), the evaluation of ratios of the zeta function (this was Conjecture \ref{RatiosConj}), a uniform upper bound on moments of ratios (this was Theorem \ref{ratio_bound}) and also the more technical tail bound for oscillatory linear statistics (this was Lemma \ref{W_markov}). We will make use of these estimates in the next section. We conclude this section by outlining a proof of a tail bound for the logarithmic derivative of a characteristics polynomial, analogous to Theorem \ref{log_bound}. Such an estimate we do not directly need in what follows, but follows easily from the others and appears to be new in the literature.

The unitary group $U(N)$ is the group of $N\times N$ complex matrices $g$ satisfying $g^\ast g = I$. In what follows we endow this group with Haar probability measure. Any such unitary matrix $g$ has $N$ eigenvalues that lie on the unit circle, which we write as $\{e^{i2\pi \theta_1}, ..., e^{i2\pi \theta_N}\}$ with $\theta_i\in[-1/2,1/2)$ for all $i$.

The $k$ level correlations of eigenvalues are in this case known exactly \cite[Eq. (39.12)]{Bu}.

\begin{theorem} [The Weyl-Gaudin-Dyson integration formula]
For $k\leq N$ and any integrable function $\eta: [-N/2,N/2)^k \rightarrow \mathbb{C}$,
$$
\bb{E}_{U(N)} \sum_{\substack{j_1,...,j_k \\ \mathrm{distinct}}} \eta(N\theta_{j_1},...,N \theta_{j_k}) = \int_{[-N/2,N/2)^k} \eta(x) \det_{k\times k}\big(K_N(x_i-x_j)\big)\,d^k x,
$$
where $K_N(x):= \frac{\sin(\pi x)}{N \sin (\pi x /N)}.$
\end{theorem}

This implies that for any integrable function $\eta:\bb{R}^k\rightarrow\bb{C}$,
$$
\bb{E}_{U(N)} \sum_{\substack{j_1,...,j_k \\ \mathrm{distinct}}} \eta(N\theta_{j_1},...,N \theta_{j_k}) \sim \int_{\bb R^k} \eta(x) \, \det_{k\times k}\big(K(x_i-x_j)\big)\,d^k x.
$$
This formula of course mirrors the GUE Conjecture, so that the points $\{\tfrac{\log T}{2\pi}(\gamma-t)\}$ may be modeled by the random points $\{N\theta_i\}$.

In fact, instead of the collection of points $\{N\theta_1,...,N\theta_N\}$, it will be even more natural to work with these points pulled back to have period $N$; that is we consider the collection of points $\bigcup_{\nu\in \bb{Z}}\{N(\theta_1+\nu),...,N(\theta_N+\nu)\}$. The reader may check that here too we have,
\begin{align}
\label{pullback_corr}
\notag \bb{E}_{U(N)} \sum_{\substack{j_1,...,j_k \\ \mathrm{distinct}}}\sum_{\nu\in\bb{Z}^k} \eta(N(\theta_{j_1}+\nu),...,N (\theta_{j_k}+\nu)) &= \int_{\bb{R}^k} \eta(x) \det_{k\times k}\big(K_N(x_i-x_j)\big)\,d^k x\\
&\sim \int_{\bb R^k} \eta(x) \, \det_{k\times k}\big(K(x_i-x_j)\big)\,d^k x.
\end{align}

We label the characteristic polynomial of a random unitary matrix $g$ in the following way:
\begin{equation}
\label{char_def}
\Lambda(A):= \det(1-e^{-A}g),
\end{equation}
where $A$ may be any complex number.

Note that
\begin{align}
\label{Lambda_product}
\notag \frac{\Lambda(\alpha/N)}{\Lambda(\beta/N)} &= \prod_{i=1}^N \frac{ e^{-\alpha/2N}\sin \pi(\theta_i + i\alpha/2\pi N)}{e^{-\beta/2N}\sin \pi (\theta_i + i\beta/2\pi N)} \\
&= e^{-(\alpha-\beta)/2}\lim_{V\rightarrow\infty}\prod_{i=1}^N \prod_{\nu=-V}^V \frac{\tfrac{\alpha}{2\pi} - i N(\theta_i+\nu)}{\tfrac{\beta}{2\pi}- iN(\theta_i+\nu)},
\end{align}
where in passing to the last line we have made use of the classical identity
$$
\sin \pi z =  \pi z \prod_{\ell=1}^\infty\Big(1-\frac{z^2}{\ell^2}\Big).
$$

Aside from being useful later on, by comparison with \eqref{hadamard}, the identity \eqref{Lambda_product} makes transparent the similarity between ratios of characteristic polynomials and ratios of the zeta function. For these ratios, we note a formula that, in effect, is due to Borodin, Olshanksi, and Strahov \cite{BoOlSt}.

\begin{theorem}
\label{ratios_unitary_finite}
For complex numbers $A_1,...,A_m$ and $B_1,...,B_m$ with $\Re\, B_\ell \neq 0$ for all $\ell$ and $A_i\neq B_j$ for all $i,j$,
$$
\bb{E}_{U(N)} \prod_{\ell=1}^m \frac{\Lambda(A_\ell)}{\Lambda(B_\ell)} = \frac{\det\Big(\frac{E(NA_i, NB_j)}{e^{A_i}-e^{B_j}}\Big)}{\det\Big(\frac{1}{e^{A_i}-e^{B_j}}\Big)}
$$
\end{theorem}

Recall that the function $E$ is defined in Conjecture \ref{RatiosConj}.

In fact, the authors in \cite{BoOlSt} do not prove exactly Theorem \ref{ratios_unitary_finite}, but rather a somewhat more general statement which may be seen with a little work to imply it. An account of this short derivation from \cite{BoOlSt} to Theorem \ref{ratios_unitary_finite} will be found in section 5.4 of the forthcoming paper \cite{ChNaNi}. There is also another proof, based on supersymmetry, in the paper \cite{KiGu}. This paper uses a rather different notation, but Theorem \ref{ratios_unitary_finite} is in fact a specialization of identity (4.35) there.

As a simple corollary,

\begin{cor}
[An asymptotic ratio evaluation]
\label{ratios_unitary_limit}
For complex numbers $\alpha_1,...,\alpha_m$ and $\beta_1,...,\beta_m$ with $\Re\, \beta_\ell \neq 0$ for all $\ell$, and $\alpha_i\neq \beta_j$ for all $i,j$,
$$
\bb{E}_{U(N)} \prod_{\ell=1}^m \frac{\Lambda(\alpha_\ell/N)}{\Lambda(\beta_\ell/N)} \sim \frac{\det\Big(\frac{E(\alpha_i, \beta_j)}{\alpha_i-\beta_j}\Big)}{\det\Big(\frac{1}{\alpha_i-\beta_j}\Big)},
$$
as $N\rightarrow\infty$.
\end{cor}

Furthermore, with a little more work,

\begin{cor}
\label{ratios_unitary_bound}
For complex numbers $\alpha,\beta$ with $\Re\, \beta \neq 0$, and for any $m\geq 0$, uniformly in $N$
$$
\bb{E}_{U(N)}\Big|\frac{\Lambda(\alpha/N)}{\Lambda(\beta/N)}\Big|^m \ll_{\alpha, \beta, m} 1.
$$
\end{cor}

\begin{proof} From H\"older's inequality, if $2k$ is an even integer larger than $m$
$$
\bb{E}_{U(N)}\Big| \frac{\Lambda(\alpha/N)}{\Lambda(\beta/N)}\Big|^m \leq \Big(\bb{E}_{U(N)}\Big|\frac{\Lambda(\alpha/N)}{\Lambda(\beta/N)}\Big|^{2k}\Big)^{m/2k}.
$$
Let $A:= \alpha/N$ and $B:=\beta/N$, and note that for a unitary matrix $g$,
\begin{align*}
\Big|\frac{\det(1-e^{-A}g)}{\det(1-e^{-B}g)}\Big|^{2k} &= \frac{\det(1-e^{-A}g)^k \det(1-e^{-\overline{A}}g^{-1})^k}{\det(1-e^{-B}g)^k \det(1-e^{-\overline{B}}g^{-1})^k} \\
&= \frac{\det(1-e^{-A}g)^k \det(1-e^{\overline{A}}g)^k}{\det(1-e^{-B}g)^k \det(1-e^{\overline{B}}g)^k}.
\end{align*}
As long as $A\neq B$, the average of this quantity can be computed exactly and seen to be uniformly bounded using Theorem \ref{ratios_unitary_finite}. And if $A=B$ the corollary is trivial.
\end{proof}

\subsection{}
We also have results that mirror Theorem \ref{quad_bound} and Lemma \ref{W_markov} for the linear statistics of (pulled-back) eigenvalues. In analogy with our discussion of zeta zeros, for a matrix $g\in U(N)$ with eigenangles $\{\theta_i\}$ as before, we use the notation
$$
\langle \eta,\, \E \rangle = \langle \eta,\, \E_N(g) \rangle := \lim_{V\rightarrow \infty} \sum_{i=1}^N \sum_{\nu = -V}^V \eta\big(N(\theta_i+\nu)\big),
$$
$$
\langle \eta,\, \E^o \rangle := \lim_{V\rightarrow\infty} \int_{-V}^V \eta(x)\,dx,
$$
$$
\langle \eta, \,\widetilde{\E}\rangle = \langle \eta,\, \widetilde{\E}_N(g) \rangle := \langle \eta,\, \E \rangle - \langle \eta,\, \E^o \rangle,
$$
when these limits exist. Clearly if $\eta$ decays quadratically the limits exist, for any unitary matrix $g$. As before, we sometime substitute $\E$ or $\E_N$ for $\E_N(g)$.

For $\eta = \hat{f}$ with $f\in L^1(\bb{R})$ and of bounded variation, the integral defining $\langle \eta, \E^o \rangle$ may be seen to converge to $(f(0+)+f(0-))/2$ (see \cite[Th. 4.3.4]{Ka} for instance). Likewise, by the Poisson summation formula (see \cite[Th. D.3]{MoVa} for instance) it may be seen for such $\eta$ that the sum defining $\langle \eta, \E \rangle$ converges also. Indeed, in this latter case the Poisson summation formula tells us that
\begin{equation}
\label{poisson1}
\langle \eta,\, \E_N(g) \rangle = \frac{1}{N}\sum_{j\in \bb{Z}} \Tr(g^j) F(-j/N),
\end{equation}
where for typographical reasons we write $F(x):= (f(x+)+f(x-))/2.$ Hence also,
\begin{equation}
\label{poisson2}
\langle \eta,\, \widetilde{\E}_N(g) \rangle = \frac{1}{N}\sum_{j\neq 0} \Tr(g^j) F(-j/N).
\end{equation}

We prove, in analogy with Theorem \ref{quad_bound},

\begin{theorem}
[A tail bound for eigenvalues] 
\label{quad_bound_unitary}
For $Q$ defined as in Theorem \ref{quad_bound}, for all $N\geq 1$ and $x\geq 2$,
$$
\bb{P}\big(\langle Q,\, \E_N\rangle \geq x\big) \ll e^{-C x\log x},
$$
where the constant $C$ and the implicit constant are absolute.
\end{theorem}

\textbf{Remark:} This result is not optimal; in fact one may show,
\begin{equation}
\label{gaussian_quad_bound_unitary}
\bb{P}\big(\langle Q,\, \E_N\rangle \geq x\big) \ll e^{-Cx^2}.
\end{equation}
This follows from a straightforward modification of the argument in \cite[Lemmas 15 and 16]{TaVu}, who are not concerned with the unitary group directly, but prove a similar estimate for the determinantal point process with sine-kernel. Nonetheless, their argument requires some knowledge of the theory of determinantal point processes, and the weaker estimate in Theorem \ref{quad_bound_unitary} will be sufficient for our purposes.

Likewise, in analogy with Lemma \ref{W_markov},

\begin{lem}[Bound on $1/\xi$ stubs for eigenvalues]
\label{W_markov_unitary}
For any $\epsilon\geq 0$ and $k\geq 1/\epsilon^2$, for all $N\geq 1$,
$$
\bb{P}(|\langle (W^{(1/k)})\,\hat{}\,,\,\E_N \rangle| \geq \epsilon) \ll \epsilon^2.
$$
\end{lem}

Indeed, these results are proved in much the same way, except that we will replace analytic number theory with a random matrix result of Diaconis and Shashahani \cite{DiSh}\footnote{Though note in this source there is a minor mistake in the statement of the result. This is corrected in, for instance, \cite{DiEv}.}.

\begin{theorem}[Diaconis-Shahshahani]
\label{Tr_moments}
Consider $a = (a_1,...,a_k)$ and $b = (b_1,...,b_k)$ with $a_1,a_2,...,b_1,b_2,... \in \bb{N}_{\geq 0}$. If
$
\sum_{j=1}^k j a_j + \sum_{j=1}^k j b_j \leq 2N,
$
then
\begin{equation}
\label{Eq_TrProd}
\bb{E}_{U(N)}\prod_{j=1}^k \Tr(g^j)^{a_j}\overline{\Tr(g^j)^{b_j}}\,dg = \delta_{ab}\prod_{j=1}^k j^{a_j} a_j!
\end{equation}
\end{theorem}

As Diaconis and Shahshahani note, if $C_1, C_2,...$ are independent standard normal complex variables (that is $C_j \law X+iY$ with $X$ and $Y$ independent and identically distributed $\mathcal{N}_\bb{R}(0,1/2)$ variables), then the right hand side of \eqref{Eq_TrProd} may also be written
\begin{equation}
\label{Eq_NProd}
\bb{E} \prod_{j=1}^k (\sqrt{j}C_j)^{a_j} \overline{(\sqrt{j}C_j)^{b_j}}.
\end{equation}
For convenience, by anology with $\Tr(g^{-j}) = \overline{\Tr(g^j)}$, we also define the random variables $C_{-j} := \overline{C_j}$, so that small moments of the traces $\Tr(g^j)$ may be identified with small moments of gaussians. (Though a caution: this identification between moments of $\Tr(g^j)$ and $\sqrt{|j|}C_j$ holds only for \emph{small} moments as in the theorem!)

We are now in a position to prove an analogue of Lemma \ref{moments}.

\begin{lem}
\label{RMTmoments}
For an absolute constant $B_2'$, uniformly for $N\geq 1$ and $2\ell\leq k$, we have
$$
\bb{E} |\langle G_k, \widetilde{\E}_N\rangle|^{2\ell} \leq (B_2' \ell)^\ell.
$$
\end{lem}

\begin{proof}
From \eqref{poisson2},
\begin{align*}
\langle G_k, \widetilde{\E}_N(g)\rangle &= \frac{1}{N} \sum_{j\neq 0} \Tr(g^j) \hat{G}_k(-j/N) \\
&= \frac{1}{N} \sum_{\substack{|j|\leq N/k \\ j\neq 0}} \Tr(g^j) \hat{G}_k(-j/N),
\end{align*}
with the second line following because $\supp \hat{G}_k \subset [-1/k,1/k]$, as in \eqref{Ghat_support}.

We have then
$$
\bb{E} |\langle G_k, \widetilde{\E}_N(g)\rangle|^{2\ell} = \bb{E} \bigg| \frac{1}{N} \sum_{\substack{|j|\leq N/k \\ j\neq 0}} \sqrt{|j|} C_j \hat{G}_k(-j/N)\bigg|^{2\ell},
$$
because one may see that any product $\prod \Tr(g^j)^{a_j} \prod \overline{\Tr(g^{j})^{b_j}}$ that would occur in the expansion of $|\sum \Tr(g^j) \hat{G}_k(-j/N)|^{2\ell}$ must have $\sum j a_j + \sum j b_j\leq 2\ell \cdot N/k \leq N$, which is certainly less than $2N$. Yet, recalling \eqref{G_hat}, we see that $\hat{G}_k$ is even, so that
\begin{align*}
\frac{1}{N} \sum_{j\neq 0} \sqrt{|j|} C_j \hat{G}_k(-j/N) &= \frac{1}{N} \sum_{j > 0} \sqrt{j} (2\Re C_j) \hat{G}_k(j/N) \\
&\law\mathcal{N}_\bb{R}\Big(0, \frac{2}{N^2} \sum_{j\geq 0} j \hat{G}_k(j/N)^2\Big),
\end{align*}
with the last reduction because the random variables $2\Re C_j$ are i.i.d real gaussians with mean $0$ and of variance $2$.

Therefore
$$
\bb{E} |\langle G_k, \widetilde{E}_N(g)\rangle|^{2\ell} = (2\ell-1)!! \Big(\frac{2}{N^2} \sum_{j\geq 0} j \hat{G}_k(j/N)^2\Big)^{\ell},
$$
with $(2\ell-1)!!:= (2\ell-1)\cdot (2\ell-3)\cdots 3\cdot 1.$ From \eqref{Ghat_support}, we know $|\hat{G}_k(x)| \leq k (1-|kx|)_+$, so
$$
\frac{2}{N^2} \sum_{j\geq 0} j \hat{G}_k(j/N)^2 \ll \frac{k^2}{N^2} \sum_{0< j < N/k} j (1- jk/N)_+^2 \ll 1.
$$
Using Stirling's formula to bound $(2\ell-1)!! = (2\ell)!/2^\ell \ell!$, we obtain the lemma.
\end{proof}

Likewise we have an analogue of Lemma \ref{pointwise_regular}.

\begin{lem}
\label{RMTpointwise_regular}
For an absolute constant $B_3'$,
$$
\langle G_k, \E^o \rangle = B_3' k.
$$
\end{lem}

\begin{proof}
This is evident from the definition of $\langle G_k, \E^o\rangle.$
\end{proof}

We now prove the tail bound for eigenvalues, Theorem \ref{quad_bound_unitary}, in the same manner that we proved Theorem \ref{quad_bound}.

\begin{proof}[Proof of Theorem \ref{quad_bound_unitary}]
For even integers $k$, and all positive $y$,
$$
\bb{P}(\langle G_k, \widetilde{\E}\rangle \geq y) \leq \frac{1}{y^k}\bb{E}|\langle G_k, \widetilde{\E}\rangle|^K \leq \frac{(B_2'k)^{k/2}}{y^k},
$$
yet
\begin{align*}
\bb{P}(\langle Q, \E\rangle \geq x) &\leq \bb{P}(\langle G_k, \widetilde{\E}\rangle + \langle G_k, \E \rangle \geq x) \\
&= \bb{P}(\langle G_k, \widetilde{E} \rangle \geq x - B_3'k).
\end{align*}
With no loss of generality, we may assume $x \geq 4B_3'$ and take $k$ to be the positive even integer satisfying $x/2B_3'-2 \leq k \leq x/2B_3'$. In particular, we have $x-B_3'k \geq x/2$ and the theorem follows, as before by combining the two lines above.
\end{proof}

Our proof Lemma \ref{W_markov_unitary}, the bound for $1/\xi$ stubs, is likewise parallel to that of Lemma \ref{W_markov}.

\begin{proof}[Proof of Lemma \ref{W_markov_unitary}]
From the Poisson summation formula \eqref{poisson1},
$$
\langle (W^{(1/k)})\,\hat{}\,, \E_N\rangle = \frac{1}{N} \sum_{j\neq 0} \Tr(g^j) W^{(1/k)}(-j/N).
$$
(Note that $(W^{(1/k)}(0+)+W^{(1/k)}(0-))/2 = 0$. This enables us to dispense with the $j=0$ term of the summand.)

As $|W^{(1/k)}(x)| \ll 1$ for all $x\in \bb{R}$ and $W^{(1/k)}(x)=0$ for $|x| \geq 1/k$, we see from Theorem \ref{Tr_moments} of Diaconis and Shashahani, as long as $k\geq 2$,
\begin{align*}
\bb{E}|\langle (W^{(1/k)})\,\hat{}\,,\E_N\rangle|^2 &= \frac{1}{N^2} \sum_{j\neq 0} |j| \cdot W^{(1/k)}(-j/N)^2 \\
&\ll \frac{1}{N^2} \sum_{|j|\leq N/k} |j| \\
&\ll \frac{1}{k^2}.
\end{align*}
Now, as in the proof of Lemma \ref{W_markov}, for $\epsilon > 1/2$, trivially,
$$
\bb{P}\big(|\langle (W^{(1/k)})\,\hat{}\,,\E \rangle | \geq \epsilon\big) \leq 1 \ll \epsilon^2.
$$
On the other hand, if $\epsilon \leq 1/2$, then the conditions of the lemma at hand force that $k\geq 2$, so that
\begin{align*}
\bb{P}\big(|\langle (W^{(1/k)})\,\hat{}\,, \E \rangle | \geq \epsilon\big) \leq & \frac{1}{\epsilon^2}\bb{E}\big(|\langle (W^{(1/k)})\,\hat{}\,,\E \rangle |^2\big) \\
\ll& \frac{1}{(\epsilon k)^2} \leq \epsilon^2.
\end{align*}
\end{proof}


\subsection{}
As with the zeta function, we can apply this technique to get tail bounds for the logarithmic derivative of the characteristic polynomial of a unitary matrix, which may be of independent interest. 

\begin{theorem}
\label{log_bound_unitary}
Fix $\alpha > 0$. For all $x \geq 2$ and $N\geq 1$,
$$
\mathbb{P}\Big(\frac{1}{N} \Big|\frac{\Lambda'}{\Lambda}\Big(\frac{\alpha}{N}\Big)\Big| \geq x \Big) \ll e^{-C x\log x},
$$
where the constant $C$ and the implied constant depend only on $\alpha$.
\end{theorem}

The proof of Theorem \ref{log_bound_unitary} follows closely that of Theorem \ref{log_bound}, and we do not require Theorem \ref{log_bound_unitary} in the remainder of this paper, so we will only indicate the main points here. Note first that much as the proof of Theorem \ref{quad_bound_unitary} follows the proof of Theorem \ref{quad_bound}, by following in turn the proof of Theorem \ref{J_bound}, one may show that
\begin{equation}
\label{J_bound_unitary}
\mathbb{P}(\langle J, \widetilde{\E}_N \rangle \geq x) \ll e^{-C x\log x},
\end{equation}
where $J$ is defined by \eqref{J_def}.

On the other hand, we will show below that
\begin{equation}
\label{log_d_0}
\frac{1}{N}\frac{\Lambda'}{\Lambda}\Big(\frac{\alpha}{N}\Big) = \langle I_\alpha, \widetilde{\E}_N\rangle.
\end{equation}

With this identity in place, exactly as in the proof of Theorem \ref{log_bound}, we note again that $I_a(\xi) = J(\xi) + O(Q(\xi))$, and therefore Theorem \ref{log_bound_unitary} follows from \eqref{log_d_0} and Theorem \ref{quad_bound_unitary}.

We turn therefore to a demonstration of \eqref{log_d_0}. A computation reveals
\begin{equation}
\label{log_d_1}
\frac{1}{N}\frac{\Lambda'}{\Lambda}\Big(\frac{\alpha}{N}\Big) = \frac{1}{N} \sum_{i=1}^N \frac{1}{e^{\alpha/N - i2\pi \theta_i}-1}.
\end{equation}
Using the expansion,
$$
\frac{1}{e^{i2\pi z}-1} = -\frac{1}{2} + \frac{1}{i2\pi} \sum_{\ell=-\infty}^\infty \frac{1}{z-\ell},
$$
(where the infinite sum is understood as a symmetric limit of partial sums), another computation reveals that the right hand side of \eqref{log_d_1} is equal to
\begin{equation}
\label{log_d_2}
-\frac{1}{2} + \sum_{i=1}^N \sum_{\ell=-\infty}^\infty \frac{1}{\alpha - i2\pi N(\theta_i + \ell)}.
\end{equation}
By definition,
$$
\sum_{i=1}^N \sum_{\ell=-\infty}^\infty \frac{1}{\alpha - i2\pi N(\theta_i + \ell)} = \langle I_\alpha, \E_N\rangle,
$$
and from computation
$$
\langle I_\alpha, \E^o \rangle = \lim_{L\rightarrow\infty}\int_{-L}^L \frac{1}{\alpha - i 2\pi \xi} \,d\xi = \frac{1}{2},
$$
so that the expression \eqref{log_d_2} is equal to $\langle I_\alpha, \widetilde{E}_N\rangle$ as claimed. This concludes our outline.

As with other bounds in this paper, probably the quantity $\Lambda'/\Lambda(\alpha/N)$ is in reality subgaussian, but we do not pursue the matter here.

\section{The average of ratios: a proof of Theorem \ref{corr_implies_ratios}}
\label{5}

\subsection{}
We begin our proof of Theorem \ref{corr_implies_ratios} by demonstrating the following proposition.

\begin{prop}
\label{GUE_to_dist}
Assume the GUE Conjecture. Then for any continuous and quadratically decaying function $\eta: \bb{R}\rightarrow \bb{R}$,
$$
\lim_{T\rightarrow\infty} \bb{E}\, e^{\langle \eta, \Z_T\rangle} = \lim_{N\rightarrow \infty} \bb{E}\, e^{\langle \eta, \E_N \rangle},
$$
with both limits existing.
\end{prop}

\begin{proof}
We note in the first place that the GUE Conjecture and the implication \eqref{pullback_corr} of the Weyl-Gaudin-Dyson integration formula imply for any non-negative integer $\ell$ and continuous and quadratically decaying function $\eta$,
\begin{equation}
\label{moment_converge}
\lim_{T\rightarrow\infty} \bb{E}\langle \eta, \Z_T\rangle^\ell = \lim_{N\rightarrow\infty} \bb{E} \langle \eta, \E_N\rangle^\ell.
\end{equation}
This is because both $\langle \eta, \Z_T\rangle^\ell$ and $\langle \eta, \E_N\rangle^\ell$ can respectively be written as a linear combination of correlation sums,
$$
\Delta_j(f_1,...,f_j):= \sum_{\substack{\gamma_1,..,\gamma_k \\ \textrm{distinct}}} f_1\Big(\frac{\log T}{2\pi}(\gamma_1-t)\Big)\cdots f_j\Big(\frac{\log T}{2\pi}(\gamma_j-t)\Big),
$$
and
$$
D_j(f_1,...,f_j):= \sum_{\nu\in\bb{Z}} \sum_{\substack{i_1,...,i_j \\ \textrm{distinct}}} f_1(N(\theta_{i_1}+\nu))\cdots f_j(N(\theta_{i_j}+\nu)),
$$
and on the GUE Conjecture $\Delta_j$ and $D_j$ have the same average as $T,N\rightarrow\infty$. For instance,
$$
\langle \eta, \Z \rangle = \Delta_1(\eta),
$$
$$
\langle \eta, \Z \rangle^2 = \Delta_1(\eta^2) + \Delta_2(\eta,\eta),
$$
$$
\langle \eta, \Z \rangle^3 = \Delta_1(\eta^3) + 3\Delta_2(\eta,\eta^2) + \Delta_3(\eta,\eta,\eta),
$$
and so on, and likewise for $\langle \eta, \E \rangle$.

Now, we note that for $x \geq 0$ and arbitrary $k\geq 0$,
\begin{equation}
\label{exp_ineq}
0 \leq e^x - \sum_{\ell=0}^k \frac{x^\ell}{\ell!} \leq \frac{x^{k+1}}{(k+1)!} e^x,
\end{equation}
as
$$
e^x - \sum_{\ell=0}^k \frac{x^\ell}{\ell!}= \sum_{\ell=k+1}^\infty \frac{x^\ell}{\ell!} \leq \frac{x^{k+1}}{(k+1)!}\sum_{j=0}^\infty \frac{x^j}{j!},
$$
with the inequality following from the relation $\frac{1}{(k+1+j)!} \leq \frac{1}{(k+1)!}\frac{1}{j!}$. Hence,
\begin{align}
\label{poly_approx}
\notag \bigg| \bb{E}\bigg(\,e^{\langle \eta, \Z \rangle} - \sum_{\ell=0}^k \frac{\langle \eta, \Z \rangle^\ell}{\ell!}\bigg)\bigg| &\leq \frac{1}{(k+1)!} \bb{E} \Big( \langle \eta, \Z \rangle^{k+1} e^{\langle \eta, \Z \rangle} \Big) \\
&\leq \frac{1}{(k+1)!} \sum_{r=0}^\infty (r+1)^{k+1} e^r\, \bb{P}\big(\langle \eta, \Z \rangle \in [r,r+1)\big)
\end{align}
Now, for $r\geq 0$, by the tail bound in Theorem \ref{quad_bound},
$$
\bb{P}\big( \langle \eta, \Z \rangle \in [r,r+1)\big) \ll e^{-Cr \log(r+2)},
$$
where the constant $C$ and the implicit constant depend on $\eta$. More trivially, from the Taylor expansion of $e^x$,
$$
(r+1)^{k+1} \leq k!\, (r+1) e^{r+1}.
$$
Applying these estimates to \eqref{poly_approx},
\begin{align*}
\bigg| \bb{E}\bigg(\,e^{\langle \eta, \Z_T \rangle} - \sum_{\ell=0}^k \frac{\langle \eta, \Z_T \rangle^\ell}{\ell!}\bigg)\bigg| &\ll \frac{1}{k+1} \sum_{r=0}^\infty (r+1)e^{2r+1} e^{-Cr \log(r+2)} \\
& \ll \frac{1}{k+1},
\end{align*}
uniformly in $T$.

By the same reasoning (replacing Theorem \ref{quad_bound} with its random matrix analogue Theorem \ref{quad_bound_unitary}),
$$
\bigg| \bb{E}\bigg(\,e^{\langle \eta, \E_N \rangle} - \sum_{\ell=0}^k \frac{\langle \eta, \E_N \rangle^\ell}{\ell!}\bigg)\bigg| \ll \frac{1}{k+1}.
$$
uniformly in $N$.

Hence, applying \eqref{moment_converge} to the above, we see that as $T\rightarrow\infty$,
$$
\bb{E}\, e^{\langle \eta, \Z_T\rangle} = \lim_{N\rightarrow\infty} \bb{E}\, e^{\langle \eta, \E_N\rangle} + O\Big(\frac{1}{k+1}\Big) + o(1).
$$
As $k$ may be chosen arbitrarily, the proposition follows.
\end{proof}

\textbf{Remark:} This theorem is only a slight modification of a standard theorem in probability theory: that the distribution of a point process is controlled by its correlation functions, provided the point process has rapidly decaying tails (c.f. \cite[Lemma 4.2.6]{HoKrPeVi}). In our context, convergence in distribution translates to the claim that if $F$ is bounded and continuous, $\lim_{T\rightarrow\infty} F(\langle \eta, \Z_T\rangle) = \lim_{N\rightarrow\infty} F(\langle \eta, \E_N \rangle)$. The fact that $e^x$ is unbounded entailed additional difficulties over the usual proof.

\subsection{}
We are finally in a position to use the GUE Conjecture to evaluate the average of ratios of the zeta function. 

\begin{proof}[Proof of Theorem \ref{corr_implies_ratios}]
Throughout this proof we take $\beta,\beta_\ell \neq 0$, and regard $m$, and $\alpha, \beta,$ $\alpha_1,...,\alpha_m,$ $\beta_1,...,\beta_m$ to be fixed, with $\alpha_i\neq \beta_j$ for all $i,j$. By \eqref{ratio_linear_stat},
$$
\exp(\langle L_{\alpha,\beta},\Z_T(t)\rangle) = (1+o(1))e^{(\alpha-\beta)/2}\frac{\zeta\Big(\frac{1}{2} + \frac{\alpha}{\log T} + it\Big)}{\zeta\Big(\frac{1}{2}+ \frac{\beta}{\log T} + it\Big)},
$$
uniformly for $t\in[T,2T]$. From this and the bound of powers of ratios, in Theorem \ref{ratio_bound}, one sees that
\begin{equation}
\label{ratio_to_linstat}
\frac{1}{T} \int_T^{2T} \prod_{\ell=1}^m \frac{\zeta\Big(\frac{1}{2} + \frac{\alpha_\ell}{\log T} + it\Big)}{\zeta\Big(\frac{1}{2}+ \frac{\beta_\ell}{\log T} + it\Big)}\,dt = \Big(\prod_{\ell=1}^m e^{(\alpha_\ell-\beta_\ell)/2}\Big)\,\bb{E} \exp(\langle \sum_{\ell=1}^m L_{\alpha_\ell,\beta_\ell}, \Z_T \rangle) + o(1).
\end{equation}
We record the observation, also following from \eqref{ratio_linear_stat}, that 
\begin{equation}
\label{abs_ratios1}
\exp(\Re\,\langle L_{\alpha,\beta},\Z_T(t)\rangle) = (1+o(1))e^{(\alpha-\beta)/2}\bigg|\frac{\zeta\Big(\frac{1}{2} + \frac{\alpha}{\log T} + it\Big)}{\zeta\Big(\frac{1}{2}+ \frac{\beta}{\log T} + it\Big)}\bigg|.
\end{equation}
This implies, of course, that the left hand side of \eqref{abs_ratios1} has a uniformly bounded $m$-th moments for fixed $\alpha, \beta,$ and $m$, with $\beta\neq 0$, by Theorem \ref{ratio_bound}.

We define
$$
L_{\alpha,\beta}^{(1/k)}(\xi):= L_{\alpha,\beta}(\xi) - i(\alpha-\beta) (W^{(1/k)})\,\hat{}\,(\xi).
$$
Intuitively, $L_{\alpha,\beta}^{(1/k)}$ should be thought of as an approximation to the function $L_{\alpha,\beta}(\xi)\mathbf{1}_{|\xi| \leq k}$. In particular, from \eqref{L_estimate} and Lemma \ref{J_approx} -- which demonstrate that both $L_{\alpha,\beta}$ and $(W^{(1/k)})\,\hat{}$ may be decomposed into a linear combination of the function $J$ and a function that decays quadratically -- we see that 
\begin{equation}
\label{L_quad_decay}
L_{\alpha,\beta}^{(1/k)}(\xi) \ll_k Q(\xi).
\end{equation}
Because $(W^{(1/k)})\,\hat{}$ is real valued, we have that for $\alpha, \beta \in \mathbb{R},$ with $\beta\neq 0$,
\begin{equation}
\label{abs_ratios2}
\exp(\Re\,\langle L_{\alpha,\beta}^{(1/k)},\Z_T(t)\rangle) = \exp(\Re\,\langle L_{\alpha,\beta},\Z_T(t)\rangle)
\end{equation}
and so the left hand side of \eqref{abs_ratios2} also has a uniformly bounded $m$-th moments for fixed $\alpha, \beta$ and $m$, with $\beta\neq 0$.

In the proof that follows we let $\epsilon > 0$ be arbitrary but small, and choose $k \geq 1/\epsilon^2$. Defining
$$
A:= \sum_{\ell=1}^m (\alpha_\ell-\beta_\ell),
$$
and returning to \eqref{ratio_to_linstat}, we have
\begin{equation}
\label{decomposition}
\bb{E} \exp(\langle \sum_{\ell=1}^m L_{\alpha_\ell,\beta_\ell}, \Z\rangle) = \bb{E} \exp(\langle \sum_{\ell=1}^m L_{\alpha_\ell,\beta_\ell}^{(1/k)} + iA (W^{(1/k)})\,\hat{}\, ,\Z\rangle).
\end{equation}
We split this average into two parts, writing
$$
H_{\geq \epsilon}:=\{t\in[T,2T]: |\langle (W^{(1/k)})\,\hat{}\,,\Z\rangle \geq \epsilon \},
$$
$$
H_{< \epsilon}:=\{t\in[T,2T]: |\langle (W^{(1/k)})\,\hat{}\,,\Z\rangle < \epsilon \}.
$$
Then \eqref{decomposition} is equal to
\begin{align*}
&\underbrace{\bb{E}\,\mathbf{1}_{H_{\geq \epsilon}}\cdot \exp(\langle \sum_{\ell=1}^m L_{\alpha_\ell,\beta_\ell}^{(1/k)} + iA (W^{(1/k)})\,\hat{}\, ,\Z\rangle)}_{:=M} \\
&+ \underbrace{\bb{E}\,\mathbf{1}_{H_{< \epsilon}}\cdot \exp(\langle \sum_{\ell=1}^m L_{\alpha_\ell,\beta_\ell}^{(1/k)} + iA (W^{(1/k)})\,\hat{}\, , \Z\rangle)}_{:=N}.
\end{align*}
For sufficiently large $T$ (depending on $\epsilon$), by Cauchy-Schwarz,
\begin{align*}
|M| &\leq \sqrt{\bb{P}(H_{\geq \epsilon})} \sqrt{\bb{E}\exp(2\Re\, \langle \sum_{\ell=1}^m L_{\alpha_\ell,\beta_\ell}^{(1/k)} ,\Z\rangle)}\\
&\ll \epsilon,
\end{align*}
with the last line following from Lemma \ref{W_markov} (our bound on $1/\xi$ stubs) to bound $\bb{P}(H_{\geq \epsilon})$ and \eqref{abs_ratios2} to bound the other term. 

On the other hand,
\begin{align*}
N &= \bb{E}\,\mathbf{1}_{H< \epsilon} \exp(\langle \sum_{\ell=1}^m L_{\alpha_\ell,\beta_\ell}^{(1/k)}, \Z\rangle  + O(\epsilon)) \\
&= \bb{E}\, \mathbf{1}_{H< \epsilon} \exp(\langle \sum_{\ell=1}^m L_{\alpha_\ell,\beta_\ell}^{(1/k)}, \Z\rangle) + O\big(\epsilon\cdot \bb{E}\, \mathbf{1}_{H< \epsilon} \exp(\Re\,\langle \sum_{\ell=1}^m L_{\alpha_\ell,\beta_\ell}^{(1/k)}, \Z\rangle),
\end{align*}
as for small $\epsilon$, we have $e^{O(\epsilon)}=1+O(\epsilon).$ Using \eqref{abs_ratios2}, we see that
$$
\bb{E}\, \mathbf{1}_{H< \epsilon} \exp(\Re\,\langle \sum_{\ell=1}^m L_{\alpha_\ell,\beta_\ell}^{(1/k)}, \Z\rangle) \ll 1,
$$
so that
\begin{align*}
N &= \Big(\bb{E}\, \mathbf{1}_{H< \epsilon} \exp(\langle \sum_{\ell=1}^m L_{\alpha_\ell,\beta_\ell}^{(1/k)}, \Z\rangle)\Big)  + O(\epsilon) \\
&= \Big(\bb{E} \exp(\langle \sum_{\ell=1}^m L_{\alpha_\ell,\beta_\ell}^{(1/k)}, \Z\rangle)  - \bb{E}\, \mathbf{1}_{H\geq \epsilon}\exp(\langle \sum_{\ell=1}^m L_{\alpha_\ell,\beta_\ell}^{(1/k)}, \Z\rangle) \Big) + O(\epsilon).
\end{align*}
And as before, for sufficiently large $T$, by Cauchy-Schwarz,\footnote{Note that it is really only in the inequalities that follow that we have exploited the assumption that $\alpha,\beta$ are real. It is from this assumption that we can easily bound $\bb{E}\,\exp(\langle \sum_{\ell=1}^m \Re\,L_{\alpha_\ell,\beta_\ell}^{(1/k)}, \Z\rangle)$ uniformly in $k$, by using that fact that $\Re i A (W^{(1/k)})\,\hat{}\, = 0$.}
\begin{align*}
\bb{E}\, \mathbf{1}_{H\geq \epsilon}\exp(\langle \sum_{\ell=1}^m L_{\alpha_\ell,\beta_\ell}^{(1/k)}, \Z\rangle) &\leq \sqrt{\bb{P}(H_{\geq \epsilon})} \sqrt{\bb{E}\,\exp(2\Re\,\langle \sum_{\ell=1}^m L_{\alpha_\ell,\beta_\ell}^{(1/k)}, \Z\rangle)} \\
&\ll \epsilon.
\end{align*}
Putting everything together, we have that
\begin{equation}
\label{linstat_approx}
\bb{E} \exp(\langle \sum_{\ell=1}^m L_{\alpha_\ell,\beta_\ell}, \Z_T\rangle) = \bb{E} \exp(\langle \sum_{\ell=1}^m L_{\alpha_\ell,\beta_\ell}^{(1/k)}, \Z_T\rangle) + O(\epsilon),
\end{equation}
uniformly for sufficiently large $T$.

In exactly the same manner, this argument may be repeated for eigenvalues of the unitary group, using the results of section \ref{4}. We see that
\begin{equation}
\label{ratio_to_linstat_unitary}
\bb{E}_{U(N)} \prod_{\ell=1}^m \frac{\Lambda(\alpha_\ell/N)}{\Lambda(\beta_\ell/N)} = \Big(\prod_{\ell=1}^m e^{(\alpha_\ell-\beta_\ell)/2}\Big)\,\bb{E} \exp(\langle \sum_{\ell=1}^m L_{\alpha_\ell,\beta_\ell}, \E_N \rangle) + o(1),
\end{equation}
in analogy to \eqref{ratio_to_linstat}, and 
\begin{equation}
\label{linstat_approx_unitary}
\bb{E} \exp(\langle \sum_{\ell=1}^m L_{\alpha_\ell,\beta_\ell}, \E_N\rangle) = \bb{E} \exp(\langle \sum_{\ell=1}^m L_{\alpha_\ell,\beta_\ell}^{(1/k)}, \E_N\rangle) + O(\epsilon),
\end{equation}
uniformly for all $N$, in analogy with \eqref{linstat_approx}.

Using \eqref{ratio_to_linstat} and \eqref{linstat_approx}, we see that
\begin{multline*}
\frac{1}{T} \int_T^{2T} \prod_{\ell=1}^m \frac{\zeta\Big(\frac{1}{2} + \frac{\alpha_\ell}{\log T} + it\Big)}{\zeta\Big(\frac{1}{2}+ \frac{\beta_\ell}{\log T} + it\Big)}\,dt \\= \Big(\prod_{\ell=1}^m e^{(\alpha_\ell-\beta_\ell)/2}\Big)\,\bb{E}\, \exp(\langle \sum_{\ell=1}^m L_{\alpha_\ell,\beta_\ell}^{(1/k)}, \Z_T\rangle) + O(\epsilon) +o(1),
\end{multline*}
as $T\rightarrow\infty$. Likewise, passing from \eqref{ratio_to_linstat_unitary} to \eqref{linstat_approx_unitary},
$$
\bb{E}_{U(N)} \prod_{\ell=1}^m \frac{\Lambda(\alpha_\ell/N)}{\Lambda(\beta_\ell/N)} = \Big(\prod_{\ell=1}^m e^{(\alpha_\ell-\beta_\ell)/2}\Big)\,\bb{E}\, \exp(\langle \sum_{\ell=1}^m L_{\alpha_\ell,\beta_\ell}^{(1/k)}, \E_N\rangle) + O(\epsilon)+o(1),
$$
as $N\rightarrow\infty$. 

Proposition \ref{GUE_to_dist} implies that the main terms on the right hand sides of these identities are asymptotically equal:
$$
\lim_{T\rightarrow\infty}\bb{E}\, \exp(\langle \sum_{\ell=1}^m L_{\alpha_\ell,\beta_\ell}^{(1/k)}, \Z_T\rangle) = \lim_{N\rightarrow\infty} \bb{E}\, \exp(\langle \sum_{\ell=1}^m L_{\alpha_\ell,\beta_\ell}^{(1/k)}, \E_N\rangle).
$$
Hence,
$$
\frac{1}{T} \int_T^{2T} \prod_{\ell=1}^m \frac{\zeta\Big(\frac{1}{2} + \frac{\alpha_\ell}{\log T} + it\Big)}{\zeta\Big(\frac{1}{2}+ \frac{\beta_\ell}{\log T} + it\Big)}\,dt = \lim_{N\rightarrow\infty}\bb{E}_{U(N)} \prod_{\ell=1}^m \frac{\Lambda(\alpha_\ell/N)}{\Lambda(\beta_\ell/N)} + O(\epsilon) + o(1).
$$
Because $\epsilon$ is arbitrary, our theorem now follows from the evaluation in Corollary \ref{ratios_unitary_limit}.
\end{proof}

\subsection{}
We have said that similar methods may be used to show that the GUE Conjecture implies not only the Local Ratio Conjecture with real translations, but in fact the Local Ratio Conjecture in general. We conclude by giving a very brief sketch of how this may be done. We note that in the above argument, the only place we have used the assumption that $\alpha_1,..,\alpha_m, \beta_1,..., \beta_m$ are real is in exploiting the fact that then $\Re i A (W^{(1/k)})\,\hat{}\, = 0$. We do note really need for this term to be $0$ though; we need only for its exponential moments to be uniformly bounded in $k$. That is, if one shows that uniformly for large $k$,
\begin{equation}
\label{W_uniform_bound}
\bb{P}(|\langle (W^{(1/k)})\,\hat{}\,, \Z \rangle | \geq x) \ll e^{-C x\log x},
\end{equation}
this is enough to bound the terms
$$
\bb{E}\,\exp(2\Re\,\langle \sum_{\ell=1}^m L_{\alpha_\ell,\beta_\ell}^{(1/k)}, \Z\rangle)
$$
uniformly in $k$, and the proof proceeds as before. \eqref{W_uniform_bound} in turn may be proven in much the same way as Theorem \ref{quad_bound}, Theorem \ref{J_bound} and Corollary \ref{JZ_bound}.

We note the converse implication, that the Local Ratios Conjecture implies the GUE Conjecture, may be derived from the combinatorial work of Conrey and Snaith \cite[Th. 8]{CoSn2}, along with a uniform bound like Theorem 1.3. Indeed, using a Tauberian argument, it should be possible to show that just the Local Ratio Conjecture with real translations also implies the GUE Conjecture, but we do not treat the matter here.



\end{document}